\documentclass[twoside,reqno,11pt,dvipsnames]{amsart}
\usepackage[a4paper,left=3cm,right=3cm,top=3cm,bottom=3cm]{geometry}

\usepackage{xcolor}                                        
\usepackage[T1]{fontenc}                                   
\usepackage[utf8]{inputenc}                                
\usepackage{microtype}                                     
\usepackage{preamble/symbols}                              
\usepackage{preamble/environments}                         
\usepackage{enumitem}                                      
\usepackage{subcaption}                                    
\usepackage{booktabs}                                      
\usepackage{hyperref}                                      
\definecolor{colorlinks}{RGB}{0, 24, 168}                  
\definecolor{colorcites}{RGB}{124, 10, 2}                  
\hypersetup{                                               
    colorlinks=true,                                       
    linkcolor=colorlinks,                                  
    citecolor=colorcites,                                  
    urlcolor=colorlinks,                                   
    pdfborder={0 0 0}                                      
}                                                          

\usepackage{preamble/boxes}                                
\usepackage{preamble/todos}                                
\makeatletter                                             
\providecommand\@dotsep{4.5}                             
\renewcommand{\listoftodos}[1][\@todonotes@todolistname]{
    \@starttoc{tdo}{#1}}                                  
\makeatother                                              
\usepackage{preamble/headings}                             
\makeatletter                                             
\def\l@section{\@tocline{1}{0pt}{0pt}{1.5pc}{}}          
\makeatother                                              

\usepackage[
    style=alphabetic,
    date=year,
    maxnames=3,
    maxbibnames=99
]{biblatex}
\AtEveryBibitem{%
    \clearfield{url}%
    \clearfield{urldate}%
    \clearfield{urlday}%
    \clearfield{urlmonth}%
    \clearfield{urlyear}%
    \clearfield{doi}%
    \clearfield{isbn}%
    \clearfield{issn}%
}

\newcommand\M{\mathbb{M}}
\newcommand\F{\mathbb{F}}
\newcommand\Fm{\mathfrak{m}}
\newcommand\FS{\mathfrak{S}}

\usepackage{accents}
\newcommand{\ubar}[1]{\underaccent{\bar}{#1}}
\newcommand\corr{\xi}
\newcommand\corrXY[1]{\corr_{\operatorname{XY}}(#1)}
\newcommand\corrHeight[1]{\corr_{\operatorname{Height}}(#1)}
\newcommand\SigCov{\operatorname{SigCov}}

\newcommand\massXY[1]{\Fm_{\operatorname{XY}}(#1)}
\newcommand\massHeight[1]{\Fm_{\operatorname{Height}}(#1)}

\newcommand\Sign[1]{\operatorname{Sign}(#1)}
\newcommand\stripXY[1]{\FS_{#1}}
\newcommand\n{{\bf n}}
\newcommand\Poisson[1]{\operatorname{Poisson}[#1]}

\newcommand\GammaLoop[1]{\Gamma_{\operatorname{Loop},#1}}

\renewcommand\S{\mathbb{S}}
\renewcommand\P{\mathbb{P}}

\newcounter{counterStrategy}

\newtheorem{strategy}[counterStrategy]{Strategy}

\title{Bijecting the BKT transition}
\author{Piet Lammers}
\address{Institut des Hautes \'Etudes Scientifiques (2020--2023)}
\email{lammers@ihes.fr}

\address{CNRS, Sorbonne Université, LPSM (2023--present)}
\email{piet.lammers@cnrs.fr}

\begin{document}

\begin{abstract}
    We consider the \emph{classical XY model} (or \emph{classical plane rotor model}) on
the two-dimensional square lattice graph as well as its dual model, which is a
model of \emph{height functions}. The XY model undergoes the
\emph{Berezinskii--Kosterlitz--Thouless transition}. A heuristic predicts that
this transition should coincide with the
\emph{localisation--delocalisation transition} for height functions: the primal
and dual models share the same partition function, and the transition in either
model should coincide with its unique point of non-analyticity as a function of
the inverse temperature. We use probabilistic arguments to prove that the
correlation length (the reciprocal of the mass) of the XY model is exactly
twice that of the height function. In particular, this confirms this prediction:
the BKT phase of the XY model coincides with the delocalised phase of the dual
height function.

\end{abstract}
\maketitle


\setcounter{tocdepth}{1}
\tableofcontents


\section{Introduction}

\subsection{Main result}

We consider the \emph{classical XY model} (or \emph{plane rotor model})
on the two-dimensional square lattice graph at \emph{inverse
temperature} $\beta\in[0,\infty)$, as well as its dual model,
which is a model of integer-valued \emph{height functions} on the faces of the
square lattice (cf.~Subsection~\ref{subsec:definitions} for a formal description).
Both models undergo a phase transition:
\begin{itemize}
    \item \textbf{XY model:} the two-point function decays \emph{exponentially fast} or
        \emph{polynomially fast} in the distance;
        this phase transition is called the \emph{Berezinskii--Kosterlitz--Thouless (BKT)
        transition},
        and the polynomial decay phase is also called the \emph{BKT regime};
    \item \textbf{Height function:} the height function is either \emph{localised} (with
            exponential decay of the covariance in the distance between the points at which
        the height is measured), or it is \emph{delocalised} (with logarithmic growth of
        the pointwise variance in the distance to the boundary).
\end{itemize}
It was not previously known whether the phase transitions of the two models coincide.

When a model exhibits exponential decay,
we define the \emph{mass} $\mathfrak{m}(\beta)\in[0,\infty]$ as the exponential decay rate
of the two-point function,
and the \emph{correlation length} $\corr(\beta):=\mathfrak{m}(\beta)^{-1}\in[0,\infty]$
as its inverse.
The mass naturally vanishes when the decay is subexponential.
The following theorem is the main result of this article.

\begin{maintheorem}
    \label{thm:main_corr}
    For any fixed inverse temperature $\beta\in[0,\infty)$,
    we have
    \begin{equation}
        \corrXY{\beta}=2\cdot\corrHeight{\beta}.
    \end{equation}
    In particular, the phase transitions of the XY model and the height function
    coincide: the XY model is in the BKT regime if and only if the height function is delocalised.
\end{maintheorem}

The theorem admits two slight extensions.
\begin{itemize}
    \item \textbf{Other graphs.}
        Although this article focuses on the square lattice graph,
        Theorem~\ref{thm:main_corr} extends, with almost the same proof, to doubly
        periodic planar graphs having at least one reflection symmetry and one
        symmetry under rotation by some angle $\alpha\in(0,\pi)$;
        see~\cite[Subsection~1.3]{Lammers_2026_DichotomyTheoryHeight}.
        In particular, our main result applies to the triangular and hexagonal lattices;
        see~\cite{EngelenburgLis_2023_ElementaryProofPhase,AizenmanHarelPeled_2022_DepinningIntegerrestrictedGaussian,Lammers_2022_HeightFunctionDelocalisation}.
    \item \textbf{The mass and correlation length in other directions.}
        In this article, the mass and correlation length are defined along the coordinate axes.
        Theorem~\ref{thm:main_corr} extends, with the same proof, when they are
        defined in any other direction.
\end{itemize}

Subsection~\ref{subsec:historical} gives background on the context of Theorem~\ref{thm:main_corr}.
Subsection~\ref{subsec:definitions} provides formal definitions of the two models.
Subsection~\ref{subsec:phtwo} describes known results used as black boxes in the proof.
Subsection~\ref{subsec:proof_overview} outlines the proof of Theorem~\ref{thm:main_corr},
which is contained in the remaining sections.
Subsection~\ref{subsec:villain_extension} describes how the result extends to the Villain model.

\subsection{Background and historical context}
\label{subsec:historical}

The Mermin--Wagner theorem rules out a conventional ordered phase in the
two-dimensional XY
model (and other continuous-symmetry spin
systems)~\cite{MerminWagner_1966_AbsenceFerromagnetismAntiferromagnetism}.
Against this background,
Berezinskii~\cite{Berezinskii_1971_DestructionLongrangeOrder,Berezinskii_1972_DestructionLongrangeOrder}
and later Kosterlitz and
Thouless~\cite{KosterlitzThouless_1973_OrderingMetastabilityPhase}
predicted that this does not preclude a different type of phase transition, now
called the
\emph{Berezinskii--Kosterlitz--Thouless transition}.
The physical mechanism proposed by Berezinskii, Kosterlitz, and Thouless is
formulated in terms of \emph{vortices} and \emph{antivortices}.
These are topological defects located on faces around which the spin
configuration makes one full turn, with the orientation distinguishing the two
types.
At high temperature they are expected to behave approximately as an independent
plasma,
leading to exponential decay of correlations.
At low temperature, the Coulomb interaction is expected to bind vortices and
antivortices into pairs, leading instead to polynomial decay of the model's two-point function.
This vortex--antivortex binding mechanism is the origin of the term
\emph{topological phase transition}.
From a thermodynamic perspective, the model's free energy
is expected to be analytic as a function of the temperature everywhere except at the
transition point,
where it is still expected to be \emph{smooth}.

Mathematically, it is known that there is a critical inverse temperature
$\beta_c\in(0,\infty)$ such that the two-point function decays polynomially precisely in
$[\beta_c,\infty)$.
This polynomial decay suggests critical behaviour in this entire interval.
The XY model's scaling limits are expected to be conformally invariant and to form a
temperature-dependent family within the broader class of conformal field
theories with central charge one; see, for example,~\cite{Ginsparg_1988_Curiosities1}.
Identifying these scaling limits is one of the main challenges in the mathematical
analysis of the BKT transition.

A common thread in the mathematical study of the BKT transition is a natural
duality (or Fourier) transform from the XY model to an integer-valued height model.
In finite volume, the two models have the same partition function.
The uncertainty principle associated with this Fourier transform makes it difficult
to relate probabilistic information across the two sides of the duality transform.
Nevertheless, the probabilistic properties of the two models should be closely related.
The fact that the phase transitions in the two models coincide is nothing more
than a very weak expression of this deeper relationship.

The first rigorous proof of BKT behaviour, expressed in terms of the two-point function's decay,
was given in the landmark work of
Fröhlich and
Spencer~\cite{FrohlichSpencer_1981_KosterlitzThoulessTransitionTwodimensional}.
Their celebrated multi-scale analysis of the Coulomb gas established
delocalised phases for integer-valued height models and the power-law decay phase
for the XY and Villain models (see also Kharash and
    Peled for a modern
exposition~\cite{KharashPeled_2017_FrohlichSpencerProofBerezinskiiKosterlitzThouless}).
Their approach is still the subject of active research:
we highlight the work of Garban and
Sepúlveda~\cite{GarbanSepulveda_2023_StatisticalReconstructionGFF}
who use methods similar to those of Fröhlich and Spencer for a statistical reconstruction problem
of the GFF given a projection.
In this context, all symmetry is destroyed by the conditional information provided by the
projection of the GFF, underlining the robustness of the original Fröhlich--Spencer method.
In~\cite{BauerschmidtParkRodriguez_2024_DiscreteGaussianModel,BauerschmidtParkRodriguez_2024_DiscreteGaussianModela},
Bauerschmidt, Park, and Rodriguez go beyond the Fröhlich--Spencer result,
deriving GFF convergence for the two-dimensional height function with square interaction
at high temperature.
The authors obtain this result by writing the integer-valued model as a perturbation
of the real-valued GFF, and then using renormalisation group methods to show that the perturbations
vanish in the scaling limit.
Now that the height function is understood at high temperature,
one may wonder whether the duality transform may eventually be used to obtain the full
scaling limit of the low-temperature spin system as described in~\cite{Ginsparg_1988_Curiosities1}.

Recently, a new line of research has emerged that offers an alternative route to proving
the existence of the BKT transition.
The geometric input was a delocalisation criterion for a class of
integer-valued height models on cubic planar
graphs~\cite{Lammers_2022_HeightFunctionDelocalisation}.
This result became one ingredient in two independent, simultaneous works proving
the existence of a BKT phase.
Van Engelenburg and
Lis~\cite{EngelenburgLis_2023_ElementaryProofPhase} developed new combinatorial methods for
the Brydges--Fröhlich--Spencer loop
representation (and a variation of it)
that enable the transfer of information between XY spin correlations and
the dual height function.
They also established the absolute-value-FKG property for the height model.
Combined with height-function delocalisation and classical correlation
inequalities, these ingredients yield an elementary proof of the existence of
the BKT phase, an explicit power-law lower bound, and the sharp
dichotomy between exponential and power-law decay.
Independently, Aizenman, Harel, Peled, and
Shapiro~\cite{AizenmanHarelPeled_2022_DepinningIntegerrestrictedGaussian}
established a relation between connectivity by height level sets and
spin--spin correlations in the dual XY and Villain models.
They used this relation to show that depinning of the integer-restricted
Gaussian field implies non-summable correlations in the Villain model, and eventually
obtained analogous results for the XY model and its dual Bessel height
function.
They also developed a new monotonicity argument that extends height-function
delocalisation from degree-three planar graphs to all doubly periodic planar
graphs, in particular to $\Z^2$.
Their use of the Regev--Stephens-Davidowitz
inequality~\cite{RegevStephens-Davidowitz_2017_InequalityGaussiansLattices}
also gave an alternative proof of infinite-volume convergence for the height model.
Together, these two works opened new routes to the BKT transition and showed
that probabilistic information can pass from a height model to its
dual spin model.
Both works imply that the delocalised phase for the height function implies
power-law decay for the XY model's two-point function;
the present work goes slightly beyond these results by deriving Theorem~\ref{thm:main_corr},
which asserts that the phases are exactly matched.

We conclude by mentioning several other important works on the BKT transition.
Chayes~\cite{Chayes_1998_DiscontinuitySpinwaveStiffness} introduced a graphical
(percolation)
representation, related it to critical properties of
the XY model, and proved the predicted discontinuity of the
spin-wave stiffness.
Dubédat and
Falconet~\cite{DubedatFalconet_2022_RandomClustersVillain}
later described how this graphical representation can be upgraded to a
model on the cable system in the spirit of Lupu~\cite{Lupu_2016_LoopClustersRandom}.
This percolation perspective is somewhat orthogonal to the random currents perspective
considered here, and may lead to further insights in the future.

We also mention the work of Falco~\cite{Falco_2013_CriticalExponentsTwo},
who obtained critical exponents and the logarithmic correction of the two-dimensional
Coulomb gas at low activity.
Although this model is not exactly the XY or Villain model considered here,
extending such a result to arbitrary activity would come close to obtaining the desired BKT picture.

Garban and Sepúlveda also studied in more detail the relation between
the Coulomb gas, the spin model, and the height function
in~\cite{GarbanSepulveda_2023_QuantitativeBoundsVortex}.
They obtained non-perturbative quantitative bounds on vortex
fluctuations in the Coulomb gas and Villain model and transferred these bounds
to the integer-valued height function.
This work also provides further detail on the general duality relations between
the different models.

\subsection{Definitions of the two models}
\label{subsec:definitions}

\subsubsection*{The XY model}

For $n\in\Z_{\geq 0}$, let $\Lambda_n:=[-n,n]^2\cap\Z^2$ denote the vertex set and
$\E(\Lambda_n):=\{\{x,y\}\subset\Lambda_n:\|y-x\|_2=1\}$ the corresponding set
of nearest-neighbour edges. The symbol $\S^1\subset\C$ is used for the unit circle.
For any $n\in\Z_{\geq 0}$ and $\beta\in[0,\infty)$, we let $\langle\blank\rangle_{n,\beta}$
denote the expectation functional of the \emph{XY model in $\Lambda_n$ at inverse
temperature $\beta$},
defined by
\begin{equation}
    \label{eq:XYdefconstantbeta}
    \langle f\rangle_{n,\beta}:=
    \frac1{Z_{n,\beta}}
    \int
    \diff\sigma
    \cdot
    f(\sigma)
    \cdot
    \exp
    \sum_{\{x,y\}\in\E(\Lambda_n)}
    \frac\beta2(\sigma_x\bar\sigma_y+\bar\sigma_x\sigma_y),
\end{equation}
where $\diff\sigma$ is the Haar measure on $(\S^1)^{\Lambda_n}$
(a probability measure) and
$Z_{n,\beta}$ the \emph{partition function},
chosen such that $\langle1\rangle_{n,\beta}=1$.

\subsubsection*{Expansion of the partition function}

The XY model has a dual height function which naturally appears when expanding
the partition function. For this construction, let $\vec\E(\Lambda_n)$ denote the
directed counterparts to edges in $\E(\Lambda_n)$. By expanding the exponentials,
we get (the notations involving $\n$ are explained below):
\begin{align}
    Z_{n,\beta}
    &=
    \int
    \diff\sigma
    \exp
    \left(
        \sum\nolimits_{xy\in\vec\E(\Lambda_n)}
        \frac\beta2\sigma_x\bar\sigma_y
    \right)
    \\&=
    \sum_{\n}
    \int
    \diff\sigma
    \prod_{xy\in\vec\E(\Lambda_n)}
    \frac{{(\beta\sigma_x\bar\sigma_y/2)^{\n_{xy}}}}
    {\n_{xy}!}
    \\
    \label{eq:standard_expansion}
    &=
    \sum_{
        \n
        :\:
        \partial\n=0
    }
    \prod_{xy\in\vec\E(\Lambda_n)}
    \frac{{(\beta/2)^{\n_{xy}}}}
    {\n_{xy}!}.
\end{align}
Here and in the sequel,
the sums over $\n$ run over all \emph{currents},
that is, functions $\n:\vec\E(\Lambda_n)\to\Z_{\geq 0}$.
For any current $\n$,
we let $\partial\n$ denote its \emph{source function}
\begin{equation}
    \partial\n:\Lambda_n\to\Z
    ,\,
    x\mapsto\sum_{y\sim x}\n_{xy}-\n_{yx}.
\end{equation}
Equation~\eqref{eq:standard_expansion} is obtained by evaluating the integral
for each term $\n$, observing that only the sourceless terms yield a
nonzero result.
Let $\M_{n,\beta/2}$ denote the measure on currents
defined by
\begin{equation}
    \label{eq:HF_def}
    \M_{n,\beta/2}[\{\n\}]:=\prod_{xy\in\vec\E(\Lambda_n)}
    \frac{{(\beta/2)^{\n_{xy}}}}
    {\n_{xy}!}.
\end{equation}
In other words, $\M_{n,\beta/2}$ is the non-normalised version of the
probability measure sampling independent $\Poisson{\beta/2}$
random variables on the directed edges.
In this formalism we get
\begin{equation}
    Z_{n,\beta}=\M_{n,\beta/2}[\{\partial\n=0\}].
\end{equation}

\subsubsection*{The dual height function}

This is also the partition function of the probability measure
$\M_{n,\beta/2}[\blank|\{\partial\n=0\}]$
defined by
\begin{equation}
    \M_{n,\beta/2}[f|\{\partial\n=0\}]:=\frac{1}{Z_{n,\beta}}\int_{\{\partial\n=0\}}f(\n)\diff\M_{n,\beta/2}(\n).
\end{equation}
Let $\F(\Lambda_n)$ denote the faces of the planar graph $(\Lambda_n,\E(\Lambda_n))$.
Write $\F_\infty$ for the outer face,
and index the remaining faces by the vertex in their lower-left corner.
For example, $\F_{(0,0)}$ is the face centred at $(1/2,1/2)$.
Whenever $\n$ is sourceless, we let $h:\F(\Lambda_n)\to\Z$
denote the \emph{height function} which assigns to each
face $\F$ the total winding number of $\n$ around $\F$.
This winding number makes sense precisely because $\n$ is sourceless.
For the avoidance of doubt, the height function is defined such that:
\begin{itemize}
    \item We always have $h(\F_\infty)=0$,
    \item If $xy$ is a directed edge with the face $\F_a$ on its left and
        $\F_b$ on its right, then
        \begin{equation}
            h(\F_a)-h(\F_b)=\n_{xy}-\n_{yx}.
        \end{equation}
\end{itemize}
For any $n$ we extend the height function to the faces $\F(\Z^2)$ of the full-plane square
lattice by setting $h(\F_a):=0$ for any $a\in\Z^2\setminus [-n,n)^2$.

\subsection{External ingredients}
\label{subsec:phtwo}
This subsection describes external results used as black boxes in the proof of
Theorem~\ref{thm:main_corr}.

\subsubsection*{The XY model}

Consider the XY measure $\langle\blank\rangle_{n,\beta}$
for fixed $\beta\in[0,\infty)$.
The Ginibre inequality~\cite{Ginibre_1970_GeneralFormulationGriffiths} implies that
this measure converges to a unique shift-invariant limit in the local convergence
topology as $n\to\infty$,
which we denote $\langle\blank\rangle_{\Z^2,\beta}$.
Of central importance is the \emph{two-point function}
\begin{equation}
    \Z^2\times\Z^2\to[0,1],\,(x,y)\mapsto \langle\sigma_x\bar\sigma_y\rangle_{\Z^2,\beta}.
\end{equation}
Moreover, it is known~\cite[Lemma~19]{EngelenburgLis_2023_ElementaryProofPhase} that for
any $n\in\Z_{\geq 0}$ and $x,y,z\in\Lambda_n$,
we have
\begin{equation}
    \langle\sigma_x\bar\sigma_z\rangle_{n,\beta}
    \geq
    \langle\sigma_x\bar\sigma_y\rangle_{n,\beta}
    \langle\sigma_y\bar\sigma_z\rangle_{n,\beta}.
\end{equation}
This inequality carries over to the $n\to\infty$ limit,
and allows us to define the \emph{mass} $\massXY{\beta}\in[0,\infty]$
of the model:
\begin{equation}
    \massXY{\beta}:=
    \lim_{k\to\infty}-\frac1k\log\langle\sigma_{(0,0)}\bar\sigma_{(k,0)}\rangle_{\Z^2,\beta}.
\end{equation}
The convergence follows from the previous inequality
and the subadditivity lemma.
The mass is non-increasing in $\beta$ due to the Ginibre inequality.
The \emph{correlation length} is defined as $\corrXY{\beta}:=1/\massXY{\beta}\in[0,\infty]$.

\begin{theorem*}[BKT
    transition~\cite{Frohlich_1982_TrivialityLphd4Theories,EngelenburgLis_2023_ElementaryProofPhase,AizenmanHarelPeled_2022_DepinningIntegerrestrictedGaussian}]
    There is some $\beta_c\in(0,\infty)$ such that:
    \begin{itemize}
        \item \textbf{Exponential decay:}
            For all $\beta\in[0,\beta_c)$
            we have $\massXY{\beta}>0$,
        \item
            \textbf{Polynomial decay:}
            For all $\beta\in[\beta_c,\infty)$ we have $\massXY{\beta}=0$,
            and
            \begin{equation}
                \forall n\in\Z_{\geq 1},\quad
                \langle\sigma_{(0,0)}\bar\sigma_{(n,0)}\rangle_{\Z^2,\beta}
                \geq \frac1{4(2n-1)}.
            \end{equation}
    \end{itemize}
\end{theorem*}

The BKT phase (polynomial decay phase) was first proved to exist in the work of
Fröhlich and Spencer~\cite{FrohlichSpencer_1981_KosterlitzThoulessTransitionTwodimensional}.
Van Engelenburg and Lis~\cite{EngelenburgLis_2023_ElementaryProofPhase}
proved the dichotomy stated above. For the exponential-decay part, they used
the finite-size criterion that a boundary sum of finite-volume correlations
below one implies exponential decay. The criterion goes back to Simon and
Lieb's work on the Ising model~\cite{Simon_1980_CorrelationInequalitiesDecay,Lieb_1980_RefinementSimonsCorrelation}.
Their XY argument iterates the Lieb--Rivasseau
inequality~\cite{Lieb_1980_RefinementSimonsCorrelation,Rivasseau_1980_LiebsCorrelationInequality}.
Duminil-Copin and Tassion later used analogous boundary sums in their sharpness
proofs for percolation and the Ising model~\cite{Duminil-CopinTassion_2016_NewProofSharpness}.

\subsubsection*{The dual height function}

Consider the law of $h$ in $\M_{n,\beta/2}[\blank|\{\partial\n=0\}]$ for fixed $\beta\in[0,\infty)$.
It was shown in~\cite{EngelenburgLis_2023_ElementaryProofPhase} that
this height function falls into a larger class of height functions studied
extensively
in~\cite{Sheffield_2005_RandomSurfaces,Lammers_2026_DichotomyTheoryHeight,LammersOtt_2024_DelocalisationAbsolutevalueFKGSolidonsolid}.
We use here some simple consequences of a standard way to view this height function:
conditional on $|h|$, the signs of $h$ behave like a ferromagnetic Ising model
with coupling constants which are explicit in terms of $|h|$.
We call this the \emph{Ising coupling} of the height function with the Ising model.
The Ising model has an associated \emph{FK--Ising percolation}.
A useful result asserts that the joint law of $|h|$ with the FK--Ising edges is increasing in $n$;
see~\cite{LammersOtt_2024_DelocalisationAbsolutevalueFKGSolidonsolid} for details
and~\cite{Lammers_2026_DichotomyTheoryHeight} for further context.

By symmetry, the law of $h$ is identical to the law of $-h$.
Thus, the height expectation of each face is zero, and
we may therefore define its covariance matrix by setting
\begin{equation}
    \Cov_{n,\beta}:\Z^2\times\Z^2\to[0,\infty)
    ,\,
    (a,b)\mapsto
    \Cov_{n,\beta}[a;b]:=\M_{n,\beta/2}[h(\F_a)h(\F_b)|\{\partial\n=0\}].
\end{equation}
The Ising coupling implies that
the covariance matrix is pointwise non-decreasing in $n$,
and tends to some shift-invariant limit $\Cov_{\infty,\beta}:\Z^2\times\Z^2\to[0,\infty]$
for any fixed $\beta$ as $n\to\infty$
(see~\cite{LammersOtt_2024_DelocalisationAbsolutevalueFKGSolidonsolid}
or~\cite[Lemma~4.1]{Lammers_2026_DichotomyTheoryHeight}).
Define the \emph{mass} $\massHeight{\beta}\in[0,\infty]$ of the height function by
\begin{equation}
    \label{eq:DefMassHeight}
    \massHeight{\beta}:=
    \lim_{k\to\infty}-\frac1k\log(\Cov_{\infty,\beta}[(0,0);(0,k)]\wedge 1).
\end{equation}
Refer to~\cite[Theorem~2]{Lammers_2026_DichotomyTheoryHeight} for convergence of this limit.
As for the XY model, the correlation length is defined as the reciprocal of the mass:
\(
    \corrHeight{\beta}:=1/\massHeight{\beta}
\).
The following theorem summarises the known results on the phase transition of this height function.

\begin{theorem*}[Dichotomy for height functions~\cite{Lammers_2026_DichotomyTheoryHeight}]
    There exists a universal constant $c\in(0,\infty)$
    such that for each fixed $\beta\in[0,\infty)$, one of the following two holds true.
    \begin{itemize}
        \item \textbf{Localisation.} The mass $\massHeight{\beta}$ is strictly positive
            and $\Cov_{\infty,\beta}<\infty$.
        \item \textbf{Delocalisation.}
            The mass $\massHeight{\beta}$ is zero, $\Cov_{\infty,\beta}\equiv\infty$,
            and
            \begin{equation}
                \exists m\in\Z_{\geq 0},
                \quad
                \forall n\in\Z_{\geq m},
                \quad
                \Cov_{n,\beta}[(0,0);(0,0)]\geq c\log n.
            \end{equation}
    \end{itemize}
\end{theorem*}

The fact that localisation (finiteness of the covariance matrix)
of the height function implies exponential decay
of the covariance matrix should be interpreted as a sharpness result.
This ingredient is truly essential as it allows us to detect the height function
phase transition by looking at the mass.
A remarkable detail lies in the fact that prior to the current article,
we could not rule out the existence of multiple transition points as we vary
$\beta$ from $0$ to $\infty$.
This is due to the lack of an equivalent of the Ginibre inequality
for height functions, which would ideally yield pointwise monotonicity of the covariance
matrix in $\beta$.
For the integer-valued Gaussian free field dual to the Villain model,
Aizenman, Harel, Peled, and Shapiro showed that the variance of any linear
combination of the heights is non-decreasing in the Villain inverse temperature
$\beta$~\cite[Proposition~2.2 and proof of Theorem~1.2]{AizenmanHarelPeled_2022_DepinningIntegerrestrictedGaussian}.
For the Bessel height function corresponding to the XY model, however, this remains open.
Theorem~\ref{thm:main_corr} gives uniqueness of the transition point through
the mass identity
$\massHeight{\beta}=2\massXY{\beta}$
and the fact that $\massXY{\beta}$ is non-increasing in $\beta$.

\subsection{Innovations and proof overview}
\label{subsec:proof_overview}

\textbf{Part~\ref{part:framework} (Sections~\ref{sec:Poisson}--\ref{sec:comb}).}
In Part~\ref{part:framework}, we develop a new framework for
analysing the random current of the XY model.
In the eyes of the author, this framework differs from
the existing literature
in two different ways.
\begin{enumerate}
    \item\textbf{Poisson points and a continuous-time exploration
        (Sections~\ref{sec:Poisson}--\ref{sec:twopoint}).}
        We interpret the expansion of the XY model as a Poisson point process on $\vec
        E\times[0,\infty)$,
        where $\vec E$ denotes the set of directed edges of the simple graph $(V,E)$
        under consideration.
        We interpret the second coordinate as \emph{time} which allows us to run a natural family
        of exploration processes.
        This exploration process can be linked to two known objects.
        \begin{itemize}
            \item After a gauge transform (Lemma~\ref{lemma:gauge}),
                these explorations generate random paths and loops closely related to
                those appearing in the analysis
                of the $\phi^4$ model in the work of Brydges, Fröhlich, and
                Spencer~\cite{BrydgesFrohlichSpencer_1982_RandomWalkRepresentation}.
                These \emph{BFS loops} allowed Fröhlich to prove triviality of the
                XY model in dimension five and higher~\cite{Frohlich_1982_TrivialityLphd4Theories},
                by seeing the model as a limit of two-component $\phi^4$ models.
                Interestingly, the paths and loops are obtained in a very different way:
                in~\cite{BrydgesFrohlichSpencer_1982_RandomWalkRepresentation},
                they are not obtained via an exploration process of Poisson points,
                but rather by expanding the Neumann series of the inverse
                of a matrix related to Gaussian integration.
                It turns out that the exploration process of Poisson points
                provides greater flexibility than the approach
                in~\cite{BrydgesFrohlichSpencer_1982_RandomWalkRepresentation}.
            \item On the other hand, the new exploration process
                is reminiscent of the \emph{backbone structure} which is of great
                use in the analysis of the Ising model: it was (to the best of the author's knowledge)
                first used by Aizenman and
                Fernández~\cite{AizenmanFernndez_1986_CriticalBehaviorMagnetization} to
                prove continuity
                of the phase transition in dimension at least four,
                and was recently applied, for example, by Aizenman and
                Duminil-Copin~\cite{AizenmanDuminil-Copin_2021_MarginalTrivialityScaling}
                to prove triviality of the scaling limit at criticality in the same context.
                \emph{A posteriori}, this backbone structure may be viewed as a special,
                discrete version
                of the continuous-time exploration process presented below.
                This discrete version works for the Ising model because every element of
                the associated group $\Z/2\Z$ is its own inverse.
                For the XY model, the associated group is $\Z$, and the continuous nature
                of the exploration process is indispensable.
        \end{itemize}
        Our continuous-time exploration process provides a unifying perspective on
        the two aforementioned objects.
        Finally, it may be viewed as the analytic counterpart to the combinatorial
        approach developed in~\cite{EngelenburgLis_2023_ElementaryProofPhase}.
        Indeed, colour switching in~\cite{EngelenburgLis_2023_ElementaryProofPhase} can be
        rephrased in our framework as recolouring of the Poisson points,
        which is a transparent operation thanks to the thinning property.
        The advantage of the analytical approach presented here is that
        we can selectively integrate out part, but not all, of the Poisson randomness,
        which makes the application more flexible (e.g., applying correlation inequalities,
        performing gluing operations, etc.).

    \item \textbf{Loop erasure (Section~\ref{sec:comb}).}
        When expanding the partition function and considering the loop-erased version
        of the random walk (coming from the exploration process of the Poisson points),
        we end up with a family of \emph{cycles} (self-avoiding loops).
        We prove that the cycle decomposition encodes the covariance matrix of the height function:
        the covariance between two faces equals the expectation
        of the number of cycles surrounding both faces
        (see Lemma~\ref{lemma:cov}).
        The fact that the cycles have a well-defined inside and outside
        is an essential feature of the planar setting.
\end{enumerate}

\textbf{Part~\ref{part:inequalities} (Sections~\ref{sec:Ginibre}--\ref{sec:corproof}).}
Using the framework outlined above and a consequence of the Ginibre inequality
first observed by Fröhlich~\cite{Frohlich_1982_TrivialityLphd4Theories},
we prove the following (informally stated) identity:
for large $n$, we have
\begin{equation}
    \log\M_{n,\beta/2}[\{
        \text{some cycle surrounds both $\F_x$ and $\F_y$}
        \}|\{\partial\n=0\}
    ]
    \approx
    2\log\langle
    \sigma_x\bar\sigma_y
    \rangle_{n,\beta},
\end{equation}
which implies the desired identity $\massHeight{\beta}=2\massXY{\beta}$.
The factor two is easily explained:
a cycle of minimal length surrounding $\F_x$ and $\F_y$ is roughly twice as long
as a path of minimal length connecting $x$ and $y$,
and the log-likelihood of seeing a path or cycle is roughly linear in
its length in the exponential decay regime.
We view this equation
as the combination of two inequalities which are proved separately.
Sections~\ref{sec:strip}--\ref{sec:heightcorrlength}
prove that $\massHeight{\beta}\leq 2\massXY{\beta}$,
while Section~\ref{sec:lowerboundXYcorr}
proves that $2\massXY{\beta}\leq\massHeight{\beta}$.
Jointly these two inequalities imply Theorem~\ref{thm:main_corr}.
Section~\ref{sec:corproof} extends this result to the Villain model (Theorem~\ref{thm2}).

Section~\ref{sec:strip} demonstrates how the path appearing in the expansion
of the two-point function may be turned into a path segment which crosses a strip,
thus lower bounding the probability of such an event.
Section~\ref{sec:loop} describes how several such strip crossings can be combined
(or ``glued together'') to form
a large loop.
The argument which glues non-overlapping strips together in such a way
that the crossings match up at their endpoints is reminiscent
of~\cite{Duminil-CopinGlazmanPeled_2020_MacroscopicLoopsLoop},
where the authors prove that the loops in the loop~$\operatorname{O}(n)$ model
are not exponentially small at the Nienhuis critical point for $n\in[1,2]$.
The loop in our article is defined in such a way that its decomposition into cycles also
contains a large cycle.
Section~\ref{sec:heightcorrlength} explains how to derive the first inequality from
the loop estimate.
Section~\ref{sec:lowerboundXYcorr} proves the second inequality and is independent of
Sections~\ref{sec:strip}--\ref{sec:heightcorrlength}.

\subsection{Extension to the Villain model}
\label{subsec:villain_extension}

The Villain model is the same model as the XY model,
except that the potential functions in the Hamiltonian change.
The model may in fact be viewed as a limit of XY models
or (equivalently) interpreted as an XY model on a metric
graph~\cite{Berezinskii_1971_DestructionLongrangeOrder,AizenmanHarelPeled_2022_DepinningIntegerrestrictedGaussian,DubedatFalconet_2022_RandomClustersVillain}.
The height function dual to this model is the discrete Gaussian model,
which is more amenable to analysis than the height function of the XY
model~\cite{AizenmanHarelPeled_2022_DepinningIntegerrestrictedGaussian,BauerschmidtParkRodriguez_2024_DiscreteGaussianModel,BauerschmidtParkRodriguez_2024_DiscreteGaussianModela}.
Theorem~\ref{thm:main_corr} extends to the Villain model.

\begin{maintheorem}
    \label{thm2}
    For any fixed inverse temperature $\beta\in[0,\infty)$,
    we have
    \begin{equation}
        \corr_{\operatorname{Villain}}(\beta)
        =2\cdot\corr_{\operatorname{Gauss}}(\beta).
    \end{equation}
    In particular, the phase transitions of the Villain model and the discrete Gaussian model
    coincide: the Villain model is in the BKT regime if and only if the discrete Gaussian model is delocalised.
\end{maintheorem}

Its proof is contained in
Section~\ref{sec:corproof}.

\part{The Poisson framework}
\label{part:framework}
\section{The Poissonian structure}
\label{sec:Poisson}

\subsection{The XY model as a \texorpdfstring{$\phi^4$}{phi4} model with deterministic radii}

This section considers a general fixed finite simple graph $G=(V,E)$.
In Equation~\eqref{eq:XYdefconstantbeta},
we defined an XY model with constant couplings $\beta\in[0,\infty)$.
For the Poissonian exploration, it is essential to consider non-constant couplings,
associated with the \emph{vertices} of the graph.
To understand the definition
it is useful to have the two-component $\phi^4$ model in mind.
In that model, each spin is described by not just an \emph{angle}
$\sigma\in\S^1\subset\C$, but also by some real-valued \emph{radius} $r\in[0,\infty)$,
so that the true spin is given by the product $\phi=r\cdot\sigma\in\C$.
Thus, our coupling strengths are encoded in terms of a \emph{radius function}
$r:V\to[0,\infty)$.
For reasons that will become clear shortly, we index our objects
by the \emph{square} of this radius function.
Thus, the XY model for a given radius function $r$ is defined by
\begin{align}
    \langle f\rangle_{G,r^2}
    &
    :=
    \frac{1}{Z_{G,r^2}}
    \int
    \diff\sigma
    \cdot
    f(\sigma)
    \cdot
    \exp
    \sum_{\{x,y\}\in E}
    \frac12(\phi_x\bar\phi_y+\bar\phi_x\phi_y)
    \\
    &
    \phantom{:}=
    \frac{1}{Z_{G,r^2}}
    \int
    \diff\sigma
    \cdot
    f(\sigma)
    \cdot
    \exp
    \sum_{\{x,y\}\in E}
    \frac{r_xr_y}2(\sigma_x\bar\sigma_y+\bar\sigma_x\sigma_y).
\end{align}
For example, if we fix $\beta\in[0,\infty)$
and set $r:\Lambda_n\to[0,\infty),\,x\mapsto \sqrt\beta$,
then we have
\begin{equation}
    \langle\blank\rangle_{(\Lambda_n,\E(\Lambda_n)),r^2}
    =
    \langle\blank\rangle_{(\Lambda_n,\E(\Lambda_n)),\beta}
    =
    \langle\blank\rangle_{n,\beta}.
\end{equation}
The construction is thus consistent with Equation~\eqref{eq:XYdefconstantbeta}.
For simplicity, we always abbreviate the graph $(\Lambda_n,\E(\Lambda_n))$
by $n$ when it is used as a subscript,
and the choice of graph is sometimes omitted entirely for brevity.

Thus, for any $R:V\to[0,\infty)$, the partition function $Z_{G,R}$
is defined by
\begin{equation}
    Z_{G,R}
    :=
    \int
    \diff\sigma
    \cdot
    \exp
    \sum_{\{x,y\}\in E}
    \frac{\sqrt{R_xR_y}}2(\sigma_x\bar\sigma_y+\bar\sigma_x\sigma_y)
    =
    \sum_{\n:\:\partial\n=0}
    \prod_{xy\in\vec E}
    \frac{(\sqrt{R_xR_y}/2)^{\n_{xy}}}{\n_{xy}!}.
\end{equation}
This notation is extended to all $R\in\R^V$
by setting $Z_{G,R}:=0$ whenever $\min R<0$.
Finally, it will be useful later to introduce a special notation for ratios of partition functions:
for any function $\tau:V\to[0,\infty)$, we write
\begin{equation}
    \label{eq:Zratio}
    Z_{G,r^2}(\tau):=Z_{G,r^2-2\tau}
    ;
    \qquad
    z_{G,r^2}(\tau):=\frac{Z_{G,r^2}(\tau)}{Z_{G,r^2}}=\frac{Z_{G,r^2-2\tau}}{Z_{G,r^2}}.
\end{equation}

\subsection{The gauge lemma}

We first state a simple identity.
This lemma is key to recovering the BFS random walk directly for the XY model.
Note that the function $J$ appearing in the statement may assign
different values to the two orientations of the same simple edge.

\begin{lemma}[Gauge lemma]
    \label{lemma:gauge}
    For any function $J:\vec E\to\C$, for any current $\n:\vec E\to\Z_{\geq 0}$,
    and for any function $g:V\to\C\setminus\{0\}$, we have
    \begin{equation}
        J^{\bf n}
        :=
        \prod_{xy\in\vec E}J_{xy}^{{\bf n}_{xy}}
        =
        g^{-\partial{\bf n}}
        \prod_{xy\in\vec E}\left(
            \frac{g_x}{g_y}J_{xy}
        \right)^{{\bf n}_{xy}}
        ;
        \qquad
        g^{-\partial{\bf n}}:=\prod_{x\in V}g_x^{-\partial{\bf n}_x}.
    \end{equation}
\end{lemma}

\begin{proof}
    The lemma is immediate from the definition of $\partial\n$.
\end{proof}

We shall keep using the notations introduced in the statement of the gauge lemma.

\begin{corollary}[Gauge expansion]
    \label{cor:gauge_expansion}
    For any radius function $r:V\to[0,\infty)$
    and any $a:V\to\Z$, we have
    \begin{equation}
        Z_{r^2}\langle\sigma^a\rangle_{r^2}
        =
        r^a
        \sum_{
            \n:\:
            \partial\n=-a
        }
        \prod_{xy\in\vec E}
        \frac{(r_x^2/2)^{\n_{xy}}}{\n_{xy}!}.
    \end{equation}
    If $r_x=0$ and $a_x\neq 0$ for some $x\in V$, the right-hand side is
    understood to be zero.
\end{corollary}

\begin{proof}
    We claim that
    \begin{multline}
        Z_{r^2}\langle\sigma^a\rangle_{r^2}
        =
        \int
        \diff\sigma
        \cdot
        \sigma^a
        \cdot
        \exp
        \sum_{\{x,y\}\in E}
        \frac{r_xr_y}2(\sigma_x\bar\sigma_y+\bar\sigma_x\sigma_y)
        \\
        =
        \sum_{
            \n
            :\:
            \partial\n=-a
        }
        \prod_{xy\in\vec E}
        \frac{(r_xr_y/2)^{\n_{xy}}}{\n_{xy}!}
        =r^a\sum_{
            \n
            :\:
            \partial\n=-a
        }
        \prod_{xy\in\vec E}
        \frac{(r_x^2/2)^{\n_{xy}}}{\n_{xy}!}.
    \end{multline}
    The first equality follows from the definition of $\langle\blank\rangle_{r^2}$,
    the second by expanding all exponentials and
    integrating each term,
    and the final one from the gauge lemma with $g=r$.
\end{proof}

\begin{figure}[p]
    \centering
    \captionsetup[subfigure]{labelfont=normalfont}
    \begin{subfigure}[b]{0.48\linewidth}
        \caption{\textbf{Expansion of the two-point function.}\\
            A configuration $\Pi$ with $\partial\n=1_y-1_x$,
            contributing to the two-point correlation function.
            The initial times $T$ may be vertex-dependent.}
        \includegraphics[width=\linewidth]{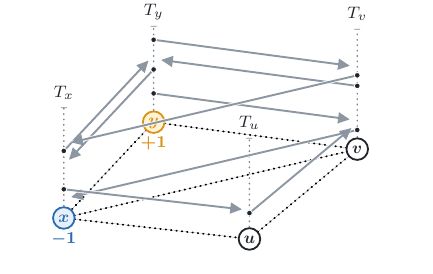}
    \end{subfigure}
    \hfill
    \begin{subfigure}[b]{0.48\linewidth}
        \caption{\textbf{One-step exploration.}\\
            At $y$, follow the Poisson edge with the highest time.
            This reduces the remaining local time at $y$, but not at the other vertices.}
        \includegraphics[width=\linewidth]{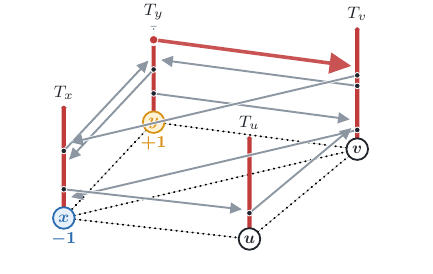}
    \end{subfigure}

    \medskip
    \begin{subfigure}[b]{0.48\linewidth}
        \caption{\textbf{Strategy I: stop when hitting $x$.}\\
            The total walk is $\omega=y\to v\to x$.
            Coloured vertical segments show the remaining times.}
        \includegraphics[width=\linewidth]{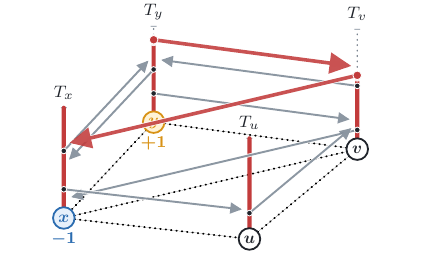}
    \end{subfigure}
    \hfill
    \begin{subfigure}[b]{0.48\linewidth}
        \caption{\textbf{Strategy II: spend $s$ at $x$, then stop.}\\
            Choose $s$ before the exploration.
            Stop once $s$ units of time have been removed at $x$.}
        \includegraphics[width=\linewidth]{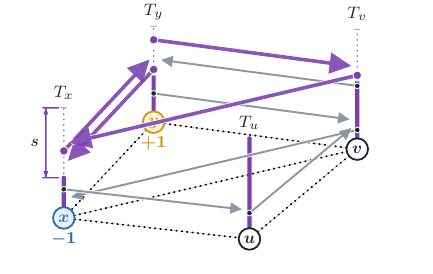}
    \end{subfigure}

    \medskip
    \begin{subfigure}[b]{0.48\linewidth}
        \caption{\textbf{Expansion of the partition function.}\\
            A sourceless configuration $\Pi$,
            contributing to the partition function.}
        \includegraphics[width=\linewidth]{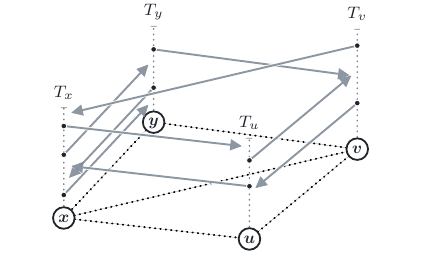}
    \end{subfigure}
    \hfill
    \begin{subfigure}[b]{0.48\linewidth}
        \caption{\textbf{The loop-erasure decomposition.}\\
            The decomposition of $\Pi$ has three cycles, each displayed with a different colour.}
        \includegraphics[width=\linewidth]{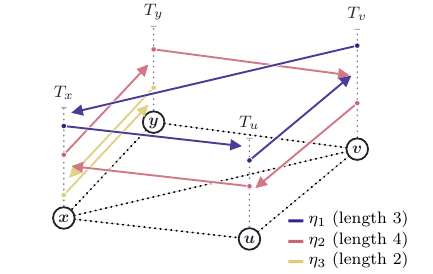}
    \end{subfigure}
    \caption{Overview of the constructions
        developed in Part~\ref{part:framework}}
    \label{fig:part1_overview}
\end{figure}
\clearpage

\subsection{Poisson interpretation}

We now cast Corollary~\ref{cor:gauge_expansion} in a more probabilistic language.

\begin{definition}[Poisson measure]
    \label{def:Poisson_local_time}
    For a \emph{local time function} $T:V\to[0,\infty)$,
    let $\lambda_T$ denote the Lebesgue measure on $\vec E\times[0,\infty)$ restricted to the set
    \begin{equation}
        \{(xy,\tau):\tau\leq T_x\}.
    \end{equation}
    Write $\|T\|:=\|\lambda_T\|=\sum_{x\in V}\deg_G x\cdot T_x$ for the total mass of this measure.
    Write $\M_{G,T}'$ for $e^{\|T\|}$ times the Poisson point process probability measure
    on $\vec E\times[0,\infty)$
    with intensity measure $\lambda_T$.
    Let $\Pi\subset\vec E\times[0,\infty)$ denote the set of Poisson edges,
    and write $\n=\n(\Pi):\vec E\to\Z_{\geq 0}$ for the random current which counts for
    each directed edge $xy$ the number of
    Poisson edges in $\{xy\}\times[0,\infty)$.
\end{definition}

The distribution of $\n$ is the same in $\M_{n,\beta/2}$ and $\M_{n,\beta/2}'$.
From now on, we shall drop the apostrophe from the notation;
we do not lose consistency with the introduction.
We also sometimes omit the reference to the graph $G$.
The following result captures Corollary~\ref{cor:gauge_expansion}
in the new formalism.

\begin{corollary}[Poisson expansion (Figures~\ref{fig:part1_overview}a and \ref{fig:part1_overview}e)]
    \label{cor:npointsigma}
    For any radius function $r:V\to[0,\infty)$
    and any $a:V\to\Z$, we have
    \begin{equation}
        Z_{r^2}\langle\sigma^a\rangle_{r^2}
        =
        r^a
        \cdot
        \M_{r^2/2}[\{\partial\n=-a\}].
    \end{equation}
    As in Corollary~\ref{cor:gauge_expansion}, the right-hand side is
    understood to be zero if $r_x=0$ and $a_x\neq0$ for some $x\in V$.
\end{corollary}

In this work, we only use this formula to expand the partition function and the two-point
correlation function.
The expansion is essentially a repeated application of the following identity,
which is a direct corollary of the definition of Poisson measures
(Definition~\ref{def:Poisson_local_time}).

\begin{lemma}[Poisson exploration]
    \label{lemma:Poisson_exploration}
    Let $G$ denote a finite simple graph,
    $T:V\to[0,\infty)$ a local time function,
    and $u\in V$ a vertex.
    Then for any $s\in [0,T_u]$, we have
    \begin{equation}
        \label{eq:Poisson_decomposition}
        \M_T
        =
        \M_{T-s\cdot 1_u}+\sum_{v\sim
        u}\int_0^s\big(\delta_{(uv,T_u-\tau)}\times\M_{T-\tau\cdot 1_u}\big)\diff\tau.
    \end{equation}
    Here $\delta_{(uv,T_u-\tau)}\times\M_{T-\tau\cdot 1_u}$ is the measure $\M_{T-\tau\cdot 1_u}$
    except that each Poisson configuration $\Pi$ is augmented by the Poisson edge $(uv,T_u-\tau)$.
\end{lemma}

\begin{proof}
    Partition the configurations according to whether there is an edge leaving
    $u$ with time coordinate in $(T_u-s,T_u]$.  Configurations with no such edge
    give the first term.  Otherwise, let $(uv,T_u-\tau)$ be the edge with
    the largest time coordinate.  Integrating over its endpoint $v$ and its time
    coordinate gives the sum on the right.
\end{proof}

Finally, this lemma may be used to expand $\M_T[\{\partial\n=-a\}]$.

\begin{lemma}[One-step exploration (Figure~\ref{fig:part1_overview}b)]
    \label{lemma:One-step exploration}
    Let $G$ denote a finite simple graph,
    $T:V\to[0,\infty)$ a local time function,
    and $u\in V$ a vertex.
    Let $a:V\to\Z$ be a function.
    Then for any $s\in[0,T_u]$, we have
    \begin{equation}
        \M_T[\{\partial\n=-a\}]
        =
        \M_{T-s\cdot 1_u}[\{\partial\n=-a\}]
        +
        \sum_{v\sim u}\int_0^s\M_{T-\tau\cdot 1_u}[\{\partial\n=-a-1_u+1_v\}]\diff\tau.
    \end{equation}
    Moreover, if $s=T_u$ and $-a_u>0$, then $\M_{T-s\cdot 1_u}[\{\partial\n=-a\}]=0$.
\end{lemma}

\begin{proof}
    The first part follows from the previous lemma by observing that
    \begin{equation}
        \big(\delta_{(uv,T_u-\tau)}\times\M_{T-\tau\cdot 1_u}\big)[\{\partial\n=-a\}]
        =
        \M_{T-\tau\cdot 1_u}[\{\partial\n=-a-1_u+1_v\}].
    \end{equation}
    For the second part, observe that if the local time at $u$ is zero,
    then $\M_{T-s\cdot1_u}$-almost every configuration has no Poisson edges leaving $u$,
    which means that $\partial\n_u\leq 0$.
    Thus, if $-a_u>0$, then the set $\{\partial\n=-a\}$ has measure zero.
\end{proof}

\section{The two-point backbone}
\label{sec:twopoint}

The purpose of this section is to find
convenient decompositions of the non-normalised
two-point function
\(Z_{r^2}\langle\sigma_x\bar\sigma_y\rangle_{r^2}\)
for $x\neq y$
in an arbitrary finite simple graph $G=(V,E)$
with $r\in[0,\infty)^V$.
By Corollary~\ref{cor:npointsigma}, we have
\begin{equation}
    \label{eq:2pstartingpoint}
    Z_{r^2}\langle\sigma_x\bar\sigma_y\rangle_{r^2}
    =
    \frac{r_x}{r_y}\cdot\M_{r^2/2}[\{\partial{\bf n}=1_y-1_x\}].
\end{equation}

\subsection{An ad hoc expansion of the two-point function}
We first decompose the measure on the right directly,
then introduce some notation to simplify the expressions.
Set $T:=r^2/2$.
The ad hoc expansion is based on the following (informally stated) exploration strategy.

\begin{strategy}[Naïve stopping rule (Figure~\ref{fig:part1_overview}c)]
    \label{strategy:naive}
    Starting at $y$, repeatedly explore the outgoing Poisson edge with
    the largest time coordinate at the current vertex.  Stop when the resulting
    walk first reaches $x$.
\end{strategy}

We now rigorously describe how this is done.
Applying Lemma~\ref{lemma:One-step exploration} with $u:=\omega_0:=y$ and $s=T_u$ yields
\begin{equation}
    \M_{r^2/2}[\{\partial{\bf n}=1_y-1_x\}]
    =\sum_{\omega_1\sim \omega_0}\int_0^{T_{\omega_0}}
    \M_{T-\tau_0\cdot 1_{\omega_0}}[\{\partial{\bf n}=1_{\omega_1}-1_x\}]\diff\tau_0.
\end{equation}
Now the terms of the sum split into two cases:
\begin{itemize}
    \item $\omega_1=x$, and then $\M_{T-\tau_0\cdot 1_{\omega_0}}[\{\partial{\bf
        n}=1_{\omega_1}-1_x\}]=Z_{r^2}(\tau_0\cdot 1_{\omega_0})$ (Corollary~\ref{cor:npointsigma}),
    \item $\omega_1\neq x$, and then $\M_{T-\tau_0\cdot 1_{\omega_0}}[\{\partial{\bf
        n}=1_{\omega_1}-1_x\}]$ can be expanded again.
\end{itemize}
Suppose that we iterate this procedure $n\in\Z_{\geq 1}$ times (including the first step above),
expanding all terms of the second type.
Then we obtain two types of terms:
\begin{itemize}
    \item Terms corresponding to walks from $y$ to $x$ whose length is at most $n$ and
        which visit $x$ only at the last step;
    \item Remainder terms corresponding to walks starting at $y$ of length $n$ that have
        not visited $x$.
\end{itemize}
Every remainder term has exposed $n$ distinct Poisson edges.  Hence their
total mass is at most $\M_{r^2/2}[\{|\Pi|\geq n\}]$, which tends to zero
because the intensity measure is finite.  Letting $n\to\infty$ therefore gives
\begin{equation}
    \label{eq:backbone_expansion}
    \M_{r^2/2}[\{\partial{\bf n}=1_y-1_x\}]
    =
    \sum_{\omega:y\to x}\int_0^\infty\diff\tau_0\cdots\int_0^\infty\diff\tau_{\ell(\omega)-1}
    Z_{r^2}\Big(\textstyle\sum_{i=0}^{\ell(\omega)-1}\tau_i\cdot 1_{\omega_i}\Big),
\end{equation}
where:
\begin{itemize}
    \item The sum runs over all walks $\omega$ from $y$ to $x$ that visit $x$ only at
        their final vertex;
    \item $\ell(\omega)$ denotes the length of the walk, whose vertices are
        indexed by $0,\ldots,\ell(\omega)$;
    \item We recall the convention that
        \begin{equation}
            Z_{r^2}\Big(\textstyle\sum_{i=0}^{\ell(\omega)-1}\tau_i\cdot 1_{\omega_i}\Big)
            =
            Z_{r^2-2\sum_{i=0}^{\ell(\omega)-1}\tau_i\cdot 1_{\omega_i}}=0
        \end{equation}
        whenever $r^2-2\sum_{i=0}^{\ell(\omega)-1}\tau_i\cdot 1_{\omega_i}$ has a
        negative coordinate,
        so that we may take $\infty$
        as the upper integration bound in each integral.
\end{itemize}
We interpret the pair $(\omega,\tau)$ as a walk on $G$,
which spends a time $\tau_i$ at each vertex $\omega_i$.

\subsection{Walks and simplex integrals}

The above expansion and variations of it will be used many times,
and therefore we introduce some notation to simplify the expressions.

\begin{definition}[Walks]
    A \emph{walk} is a finite sequence of vertices
    $\omega=(\omega_k)_{0\leq k\leq n}\in V^{n+1}$
    such that consecutive vertices are neighbours.
    The length of such a walk $\omega$ is denoted $\ell(\omega):=n$.
    For any finite vertex sequence $\eta$, including the empty sequence,
    let $k(\eta):V\to\Z_{\geq 0}$ count how often $\eta$ visits each vertex:
    \begin{equation}
        k(\eta)_x:=\#\{i:\eta_i=x\}.
    \end{equation}
    Finally, $\omega^*:=(\omega_k)_{0\leq k<\ell(\omega)}$ denotes the
    sequence obtained by removing the final vertex of $\omega$; it is empty
    when $\ell(\omega)=0$.
\end{definition}

\begin{definition}[Sets of walks]
    Let $\Gamma$ denote the set of walks through $G$,
    and write
    \begin{align}
        &\Gamma_{xy}:=\{\omega\in\Gamma:\omega_0=x,\,\omega_{\ell(\omega)}=y\};\\
        &\Gamma_{xy}^{(a)}:=\{\omega\in\Gamma_{xy}:k(\omega)_y=a\};\\
        &\Gamma_{xy}[A]:=\{\omega\in\Gamma_{xy}:\omega^*\subset A\};
    \end{align}
    for any $x,y\in V$, $a\in\Z_{\geq 0}$, and $A\subset V$.
    For example, in this notation, $\Gamma_{xy}^{(1)}=\Gamma_{xy}[\{y\}^c]$.
\end{definition}

\begin{definition}[Simplex measures]
    Define measures $(\rho_k)_{k\in\Z_{\geq 0}}$ on $[0,\infty)$ as follows.
    The measure $\rho_0$ is the Dirac measure at $0$, while
    \begin{equation}
        \diff\rho_k(t)=\frac{t^{k-1}}{(k-1)!}\diff t,
        \qquad k\geq1.
    \end{equation}
    In particular, $\rho_1$ is Lebesgue measure.
    For any function $k:V\to\Z_{\geq 0}$,
    let $\rho_k$ denote the product measure on $[0,\infty)^V$ defined by
    \begin{equation}
        \rho_k:=\bigotimes_{u\in V}\rho_{k_u}.
    \end{equation}
    For any finite vertex
    sequence $\eta$, write $\rho_\eta:=\rho_{k(\eta)}$.
\end{definition}

The pushforward of Lebesgue measure on $[0,\infty)^{\ell(\omega)+1}$
under
\begin{equation}
    (t_0,\ldots,t_{\ell(\omega)})
    \mapsto
    \sum_{i=0}^{\ell(\omega)}t_i1_{\omega_i}
\end{equation}
is precisely $\rho_\omega$.  Thus, $\rho_\omega$ records the total time
spent at each vertex by the walk $\omega$.
Moreover, the definitions above imply that for any $n,m:V\to\Z_{\geq 0}$
and any integrable function $f:[0,\infty)^V\to\C$,
we have
\begin{equation}
    \label{simplex}
    \tag{Simplex}
    \iint f(\tau_1+\tau_2)\diff\rho_n(\tau_1)\diff\rho_m(\tau_2)
    =
    \int f(\tau)\diff\rho_{n+m}(\tau).
\end{equation}
\stepcounter{equation}

\begin{lemma}[Naïve exploration (Figure~\ref{fig:part1_overview}c)]
    \label{lemma:naive_exploration}
    Let $G$ be a finite graph, let $r:V\to[0,\infty)$, and let $x,y\in V$.
    Then
    \begin{equation}
        \label{2p.naiveM}
        \M_{r^2/2}[\{\partial{\bf n}=1_y-1_x\}]
        =
        \sum_{\omega\in\Gamma_{yx}^{(1)}}\int
        Z_{r^2}(\tau)\diff\rho_{\omega^*}(\tau)
        ,
    \end{equation}
    and
    \begin{equation}
        \tag{2p.i}
        \label{2p.naive}
        \langle\sigma_x\bar\sigma_y\rangle_{r^2}
        =
        \frac{r_x}{r_y}
        \sum_{\omega\in \Gamma_{yx}^{(1)}}
        \int
        z_{r^2}(\tau)
        \diff\rho_{\omega^*}(\tau).
    \end{equation}
    \stepcounter{equation}
\end{lemma}
\begin{proof}
    If $x\neq y$, the first equation is Equation~\eqref{eq:backbone_expansion}
    rewritten in the new notation.  If $x=y$, then
    $\Gamma_{xx}^{(1)}$ contains only the length-zero walk, so both sides equal
    $Z_{r^2}$.
    The second equation follows from Equation~\eqref{eq:2pstartingpoint},
    recalling Equation~\eqref{eq:Zratio}.
\end{proof}

\subsection{Symmetrised expansion of the two-point function}

The left-hand side of Equation~\eqref{2p.naive} is obviously symmetric under
interchanging $x$ and $y$.
The right-hand side would not be obviously symmetric without this formula:
the sum runs over walks which may visit $y$ many times, but which visit $x$ only once,
and the ratio $r_x/r_y$ introduces a further asymmetry.
In this subsection, we prove the following variation of Lemma~\ref{lemma:naive_exploration}.

\begin{lemma}[Symmetrised exploration]
    \label{lemma:symmetric_exploration}
    Let $G$ be a finite graph, let $r:V\to[0,\infty)$, and let $x,y\in V$.
    Then
    \begin{equation}
        \label{2p.symmetricM}
        \M_{r^2/2}[\{\partial{\bf n}=1_y-1_x\}]
        =
        \frac{2}{r_x^2}
        \sum_{\omega\in\Gamma_{yx}}
        \int
        Z_{r^2}(\tau)\diff\rho_{\omega}(\tau)
        ,
    \end{equation}
    and
    \begin{equation}
        \tag{2p.ii}
        \label{2p.integrated}
        \langle\sigma_x\bar\sigma_y\rangle_{r^2}
        =
        \frac{2}{r_xr_y}
        \sum_{\omega\in \Gamma_{yx}}
        \int
        z_{r^2}(\tau)
        \diff\rho_{\omega}(\tau).
    \end{equation}
    \stepcounter{equation}
\end{lemma}

Notice that the right-hand side of Equation~\eqref{2p.integrated}
(which follows from Equation~\eqref{2p.symmetricM}) is obviously symmetric under
interchanging $x$ and $y$:
walk reversal is a bijection from $\Gamma_{yx}$ to $\Gamma_{xy}$, and $\rho_\omega$
is invariant under this bijection.

To obtain this lemma, we use the following stopping rule.

\begin{strategy}[Prescribed local time at the target (Figure~\ref{fig:part1_overview}d)]
    \label{explo:II}
    Fix $s\in[0,r_x^2/2]$.  Start with Strategy~\ref{strategy:naive},
    but after first reaching $x$, continue the exploration until the total
    local time removed at $x$ is exactly $s$.
\end{strategy}

We return to Equation~\eqref{2p.naiveM},
which may be written as
\begin{equation}
    \M_{r^2/2}[\{\partial{\bf n}=1_y-1_x\}]
    =
    \sum_{\omega\in\Gamma_{yx}^{(1)}}\int
    1[\tau\leq r^2/2]\cdot
    \M_{r^2/2-\tau}[\{\partial\n=0\}]\diff\rho_{\omega^*}(\tau).
\end{equation}
Although the integrand is a partition function,
we may continue the expansion by applying Lemma~\ref{lemma:One-step exploration} to it.
Let us apply this lemma with $u=x$ and $s\in[0,r_x^2/2]$.
This yields two terms: one sourceless term where the remaining local time at $x$ is
exactly reduced by $s$,
and one term where we found an outgoing Poisson edge from $x$ with a local time at least
$r_x^2/2-s$.
This expansion may be repeatedly applied to each term:
\begin{itemize}
    \item If the term has nontrivial sources, then we apply the expansion to push the
        source forward until it is back at $x$,
    \item If the term is sourceless and the local time at $x$ is more than $r_x^2/2-s$,
        then we apply the expansion to push the local time at $x$ forward until it is
        exactly $r_x^2/2-s$,
    \item If the local time at $x$ is exactly $r_x^2/2-s$, then we do not apply the
        expansion to this term.
\end{itemize}
The resulting identity is recorded in the following lemma.

\begin{lemma}[Exploration with arbitrary stopping time (Figure~\ref{fig:part1_overview}d)]
    \label{lemma:exploration_arbitrary}
    Let $G$ be a finite graph, let $r:V\to[0,\infty)$, let $x,y\in V$, and let $s\in[0,r_x^2/2]$.
    Then
    \begin{equation}
        \label{2p.arbitraryM}
        \M_{r^2/2}[\{\partial{\bf n}=1_y-1_x\}]
        =
        \sum_{\omega\in \Gamma_{yx}}
        \int_{\{\tau_x\leq s\}}1[\tau\leq r^2/2]\cdot
        Z_{r^2}(\tau\vee (s\cdot 1_x))
        \diff\rho_{\omega^*}(\tau)
        ,
    \end{equation}
    and
    \begin{equation}
        \tag{2p.iii}
        \label{2p.anytime}
        \langle\sigma_x\bar\sigma_y\rangle_{r^2}
        =
        \frac{r_x}{r_y}
        \sum_{\omega\in \Gamma_{yx}}
        \int_{\{\tau_x\leq s\}}
        z_{r^2}(\tau\vee (s\cdot 1_x))
        \diff\rho_{\omega^*}(\tau).
    \end{equation}
    \stepcounter{equation}
\end{lemma}

\begin{proof}
    Notice that, at the very last application of the one-step exploration,
    the local time at $x$ is set to exactly $s$.
    This does not introduce an extra integral, but sets the total
    local time to $\tau\vee(s\cdot 1_x)$, where $\tau$ is the local time
    coming from all the other steps.
\end{proof}

\begin{proof}[Proof of Lemma~\ref{lemma:symmetric_exploration}]
    Equation~\eqref{2p.symmetricM} is obtained by averaging
    Equation~\eqref{2p.arbitraryM} over $s\in[0,r_x^2/2]$.
    More precisely, we integrate this formula over $s$ from $0$ to $r_x^2/2$
    with respect to Lebesgue measure, then divide by $r_x^2/2$.
    The Lebesgue integral over
    $s$ combines with the $x$-coordinate of $\rho_{\omega^*}$ to give the
    $x$-coordinate of $\rho_\omega$.  Thus the average equals the right-hand
    side of Equation~\eqref{2p.symmetricM}.  Equation~\eqref{2p.integrated}
    follows after multiplying by $r_x/r_y$ and dividing by $Z_{r^2}$.
\end{proof}

\begin{remark}
    Equation~\eqref{2p.integrated} is the form of the two-point function appearing
    in~\cite{BrydgesFrohlichSpencer_1982_RandomWalkRepresentation}.  Here it is
    realised directly for the XY model through a probabilistic exploration with a
    random stopping time.  The naïve and fixed-stopping-time expansions will also
    be used below.
\end{remark}

\section{The loop erasure decomposition}
\label{sec:comb}

We work on an arbitrary finite simple graph $G=(V,E)$ until further notice.
Let $T:V\to [0,\infty)$ denote a local time function.
Here and in the sequel,
we shall write $\P_{G,T}$ for the probability measure
\begin{equation}
    \P_{G,T}:=\M_{G,T}[\blank|\{\partial\n=0\}],
\end{equation}
and we let $\E_{G,T}$ denote the corresponding expectation functional
(with the normalisation in the conditional law understood).
Consider a sample $\Pi$ from $\P_{G,T}$.
The times of all Poisson edges are distinct almost surely,
and we only consider such samples $\Pi$.
For any set $A$ of Poisson edges, let $V(A)\subset V$ be the set of vertices
incident to an edge in $A$.
A \emph{cycle} of $\Pi$ is a nonempty subset $\eta\subset\Pi$ such that
$(V(\eta),\eta)$ is connected and every vertex in this directed multigraph
has in-degree $1$ and out-degree $1$.

\subsection{Partition into cycles via loop erasure (Figures~\ref{fig:part1_overview}e--f)}
It turns out that $\Pi$ has a natural partition into cycles.
This partition is called the \emph{loop erasure decomposition}.
We will use its properties throughout the rest of the paper.

\subsubsection*{Informal description of the loop erasure decomposition}

Consider Strategy~\ref{explo:II}, started at some vertex $x$.
Since $\partial\n=0$, we may explore until no local time at $x$ is left,
which generates a walk in $\Gamma_{xx}$.
Each step of the walk corresponds to a distinct edge in $\Pi$.
This walk is generated as follows: at each step, we choose the unused outgoing
Poisson edge from the current vertex with the largest time.

This walk in $\Gamma_{xx}$ generates a family of disjoint (as sets) cycles in $\Pi$ as follows.
Each time the walk hits its own past, we erase the loop just created and
add that loop to the decomposition.
This generates a partition of some subset of $\Pi$ into cycles,
but it does not necessarily exhaust all of $\Pi$.

The remaining cycles are generated by repeating the procedure on the unused
Poisson edges, starting from other vertices.

This is essentially Wilson's algorithm~\cite{Wilson_1996_GeneratingRandomSpanning}
for generating a uniform spanning tree and its loop erasure;
only the underlying mechanism for generating the walk is different.

\subsubsection*{Formal definition of the loop erasure decomposition}
Although the loop erasure procedure is quite standard in probability theory,
we now give an explicit and unambiguous construction of the loop erasure decomposition.
In particular, it does not depend on the order in which starting vertices are chosen.

Suppose that $\Pi$ is nonempty.
Since $\partial\n(\Pi)=0$, the in-degree and out-degree of every vertex in
$V(\Pi)$ are equal and positive.
Let $M(\Pi)\subset\Pi$ contain, for each vertex $x\in V(\Pi)$,
the outgoing Poisson edge with the largest time.
Observe that each such Poisson edge points towards another vertex in $V(\Pi)$.
Thus, $(V(\Pi),M(\Pi))$ is a finite directed graph with edge times in which
each vertex has out-degree one.
Its edge set contains a nonempty family of vertex-disjoint directed cycles;
every remaining edge eventually flows into one of these cycles.
Write $C^1(\Pi)$ for this vertex-disjoint family of cycles and
$\cup C^1(\Pi)\subset\Pi$ for its union.
Observe that $\partial\n(\cup C^1(\Pi))=0$
and therefore $\partial\n(\Pi\setminus \cup C^1(\Pi))=0$.
For consistency, we also define $C^1(\varnothing):=\varnothing$.
The set $C^1(\Pi)$ may be thought of informally as containing the cycles
of $\Pi$ in the topmost layer.
We iterate this construction by defining
\begin{equation}
    C^{k+1}(\Pi):=C^k(\Pi)\cup C^1(\Pi\setminus\cup C^k(\Pi))
    \qquad
    \forall k\in\Z_{\geq 1}.
\end{equation}
The set $C^{k}(\Pi)\setminus C^{k-1}(\Pi)$ may be thought of informally
as being the $k$-th layer of cycles.
Observe that $(C^k(\Pi))_{k\in\Z_{\geq 1}}$ is increasing and that
$C^k(\Pi)$ is a partition of $\Pi$ for all sufficiently large $k$.
Write $C(\Pi):=\lim_{k\to\infty}C^k(\Pi)$ for this partition,
called the \emph{loop erasure decomposition}.

\subsection{Activation times and the induced partial order}
If $\eta$ is a cycle and if $x\in V(\eta)$,
then write $a_x(\eta)$ for the time on the unique edge in $\eta$
pointing out of $x$.
We call $a_x(\eta)$ the \emph{activation time} of $\eta$ at $x$.

\begin{definition}
    For any partition $P$ of $\Pi$ into cycles,
    we define the relation $\preceq$ on $P$
    such that
    $\eta\preceq\eta'$ if and only if
    there exists a sequence of cycles $(\eta^k)_{0\leq k\leq n}\subset P$
    from $\eta^0=\eta$ to $\eta^n=\eta'$
    such that
    \begin{equation}
        \forall 0\leq k<n,\quad
        \exists x\in V(\eta^k)\cap V(\eta^{k+1}),
        \quad
        a_x(\eta^k)\leq a_x(\eta^{k+1}).
    \end{equation}
    We write $\preceq_P$ for $\preceq$ when considering multiple partitions
    in order to avoid confusion.
    We call $P$ a \emph{proper cycle decomposition} of $\Pi$ if and only if
    the relation $\preceq$ is a partial order.
\end{definition}

\begin{lemma}
    \label{lemma:unique_pcd}
    The loop erasure decomposition $C(\Pi)$ is the unique proper cycle
    decomposition of $\Pi$.
\end{lemma}

\begin{proof}
    We first check that $\preceq$ is a partial order on $C(\Pi)$.
    Reflexivity and transitivity of $\preceq$ follow directly from its
    definition.  We check antisymmetry on $C(\Pi)$.
    It suffices to rule out the existence of distinct cycles $\eta,\eta'\in C(\Pi)$
    such that simultaneously $\eta\preceq\eta'$ and $\eta'\preceq\eta$.
    If $\eta\preceq\eta'$ and $\eta\neq\eta'$, then, after deleting
    consecutive repetitions, every step in a chain witnessing this relation
    moves to a cycle removed in an earlier layer.
    Thus $\eta'$ has a strictly smaller layer index than $\eta$, which rules
    out $\eta'\preceq\eta$.

    Let $P$ denote any proper cycle decomposition of $\Pi$.
    We claim that $P=C(\Pi)$.
    We induct on $|\Pi|$.
    The case $\Pi=\varnothing$ is trivial, so assume that $\Pi$ is nonempty.
    Let $\eta\in P$ denote a cycle which is $\preceq_P$-maximal.
    Such a cycle must exist since $\Pi$ and $P$ are finite.
    If an edge of $\Pi$ leaving a vertex of $\eta$ had a larger time than the
    corresponding edge of $\eta$, its cycle in $P$ would be strictly above
    $\eta$.  Maximality therefore gives $\eta\subset M(\Pi)$.
    But then $\eta$ is a cycle of $(V(\Pi),M(\Pi))$,
    and therefore $\eta\in C(\Pi)$.
    Both $P\setminus\{\eta\}$ and $C(\Pi)\setminus\{\eta\}$ are proper
    cycle decompositions of $\Pi\setminus\eta$: removing one element from
    either decomposition only restricts its partial order.  By induction they
    are both equal to $C(\Pi\setminus\eta)$.
    This proves the claim.
\end{proof}

\subsection{Symmetries of the random loop erasure decomposition}

Let $\tilde\Pi$ denote the \emph{time inversion} of $\Pi$,
defined by
\begin{equation}
    \tilde\Pi:=\{(xy,T_x-\tau):(xy,\tau)\in\Pi\}.
\end{equation}
The same notation is used for the time inversion of any cycle $\eta\subset\Pi$.

\begin{proposition}
    \label{pro:time_inversion}
    \leavevmode
    \begin{enumerate}
        \item Time inversion commutes with the loop erasure decomposition, that is,
            \begin{equation}
                C(\tilde\Pi)=\{\tilde\eta:\eta\in C(\Pi)\}.
            \end{equation}
        \item The configurations $\Pi$ and $\tilde\Pi$ have the same law under
            $\P_{G,T}$.
    \end{enumerate}
\end{proposition}

\begin{proof}
    Write $P:=\{\tilde\eta:\eta\in C(\Pi)\}$,
    which is a partition of $\tilde\Pi$.
    At each vertex, time inversion reverses the order of the times of the
    outgoing edges.
    Hence, for any $\eta,\eta'\in C(\Pi)$,
    \begin{equation}
        \tilde\eta \preceq_P \tilde\eta'
        \iff
        \eta' \preceq_{C(\Pi)}\eta.
    \end{equation}
    Thus $P$ is a proper cycle decomposition of $\tilde\Pi$, so
    Lemma~\ref{lemma:unique_pcd} yields $P=C(\tilde\Pi)$.

    The time inversion map preserves the intensity measure $\lambda_T$ of the Poisson
    point process.  It also leaves the current $\n$ unchanged and therefore
    preserves the event $\{\partial\n=0\}$.  This proves the second statement.
\end{proof}

Define the \emph{directional inverse} $\eta^{-1}$
of any cycle $\eta$ by
\begin{equation}
    \eta^{-1}:=\{(yx,a_y(\eta)):(xy,\tau)\in\eta\}\subset\vec E\times[0,\infty).
\end{equation}
In other words, $\eta^{-1}$ is obtained from $\eta$ by reversing all directed
edges and updating their times so that the two cycles have the same activation
time at every vertex.

We write
\begin{equation}
    \bar C(\Pi):=\big\{\{\eta,\eta^{-1}\}:\eta\in C(\Pi)\big\}
\end{equation}
for the corresponding unoriented loop erasure decomposition.

\begin{proposition}
    \label{propo:flips}
    \leavevmode
    \begin{enumerate}
        \item Direction inversion commutes with the loop erasure decomposition, that is,
            \begin{equation}
                C\big(
                    \big(\Pi\setminus\cup S\big)
                    \cup
                    \big(\cup_{\eta\in S}\eta^{-1}\big)
                \big)
                =
                (C(\Pi)\setminus S)\cup \{\eta^{-1}:\eta\in S\}
            \end{equation}
            for all $S\subset C(\Pi)$.
        \item The law of $\Pi$ under $\P_{G,T}$ is invariant under independently replacing each
            cycle $\eta\in C(\Pi)$ by $\eta^{-1}$ with probability $1/2$.
    \end{enumerate}
\end{proposition}

\begin{proof}
    We first prove the first statement when $S=\{\eta\}$.
    Set
    \begin{equation}
        P:=(C(\Pi)\setminus\{\eta\})\cup\{\eta^{-1}\}.
    \end{equation}
    Since direction inversion preserves the set of vertices visited
    and the activation time at each such vertex, the relation $\preceq$
    is the same on $C(\Pi)$ and $P$
    once we identify $\eta$ and $\eta^{-1}$.
    Thus $P$ is a proper cycle decomposition of
    $(\Pi\setminus\eta)\cup\eta^{-1}$, and
    Lemma~\ref{lemma:unique_pcd} gives
    \begin{equation}
        C((\Pi\setminus\eta)\cup\eta^{-1})=P.
    \end{equation}
    Applying this identity successively to the cycles in $S$ proves the first
    statement for every $S\subset C(\Pi)$.

    Fix a deterministic measurable rule $A$ which, given $\bar C(\Pi)$,
    selects some of its non-singleton elements, and let $\Phi_A(\Pi)$ be
    obtained by reversing the selected cycles.  By the first statement,
    \begin{equation}
        \bar C(\Phi_A(\Pi))=\bar C(\Pi).
    \end{equation}
    Consequently, the same cycles are selected when $\Phi_A$ is applied a
    second time, so $\Phi_A$ is an involution.

    We next check that this involution preserves the measure.  This is easy to see:
    $\Phi_A$ preserves the intensity measure,
    as well as the event $\{\partial\n=0\}$.
    Therefore it preserves $\P_{G,T}$.

    Finally, conditionally on $\Pi$, independent fair coins select a uniformly
    random subset of the non-singleton elements of $\bar C(\Pi)$.  Averaging
    the preceding invariance over all selections proves the second statement.
    A two-edge cycle with $\eta=\eta^{-1}$ is unaffected by its coin.
\end{proof}

\subsection{Relation to exploration processes}

We now make a link between the explorations discussed in Section~\ref{sec:twopoint}
and the cycles in $C(\Pi)$.

\begin{proposition}
    \label{pro:loop_decomp_lemma}
    Fix a set of Poisson edges $\Pi$ satisfying $\partial\n(\Pi)=0$, and some
    vertex $x\in V$.
    Let $\omega\in\Gamma_{xx}$ be a nontrivial walk obtained by repeatedly
    taking the unused outgoing edge with the largest time and eventually
    stopping at $x$.
    Let $\eta\in C(\Pi)$ denote the cycle containing
    the last outgoing Poisson edge from $x$
    that $\omega$ uses.
    Then the set of Poisson edges $P\subset\Pi$ used by $\omega$
    is precisely
    \begin{equation}
        P=\cup H,
        \qquad
        H:=\{\eta'\in C(\Pi):\eta\preceq_{C(\Pi)}\eta'\}.
    \end{equation}
\end{proposition}

\begin{proof}
    At each step, the walk $\omega$ takes an unused outgoing edge with the largest time.
    The relation between $\omega$ and the cycles in $C(\Pi)$ can be seen through
    the following recursive construction.
    First, let $p^0$ denote the concatenation, in decreasing order of their
    activation times at $x$, of all cycles in $C(\Pi)$ through $x$ with
    activation time at least $a_x(\eta)$.
    Each cycle in this concatenation is viewed as starting and ending at $x$.
    Now consider $p^0_1$, the first vertex of $p^0$ visited after $x$.
    Consider all cycles in $C(\Pi)$ that:
    \begin{itemize}
        \item Were not chosen before,
        \item Visit $p^0_1$,
        \item Have an activation time at $p^0_1$ of at least the time of the outgoing
            edge of $p^0$ from $p^0_1$.
    \end{itemize}
    Now define the new path $p^1$ as the concatenation of the following paths:
    \begin{itemize}
        \item The path $p^0$ up to $p^0_1$,
        \item All cycles found above, concatenated in decreasing order of their
            activation times at $p^0_1$,
        \item The path $p^0$ after $p^0_1$.
    \end{itemize}
    Now move to the second step $p^1_2$ of the updated path and repeat the same procedure,
    yielding a new path $p^2$.
    Continuing in this way produces a finite sequence of paths because no cycle
    is inserted more than once.  When the final vertex $x$ is reached, no new
    cycle is inserted because all relevant cycles through $x$ were added at the
    beginning.  By the decreasing activation-time order, the next edge at each
    step is precisely the unused outgoing edge with the largest time.
    The resulting path is therefore $\omega$.

    The cycles initially inserted are $\eta$ and the cycles directly above it
    through $x$.  Each later insertion adds precisely cycles which are directly
    above a cycle already present, at the vertex where the insertion is made.
    Iterating this rule generates the transitive upward closure of $\eta$ under
    $\preceq_{C(\Pi)}$, which is exactly $H$.
\end{proof}

\subsection{Relation to the covariance matrix}
Now specialise to the square lattice.
Every cycle of length at least four traces a simple lattice circuit and hence
has a well-defined \emph{orientation} (clockwise or counterclockwise) and
\emph{interior}.  A two-edge cycle encloses no face and contributes zero to
the height function.  We say that a cycle \emph{surrounds} a face when the
face lies in its interior.  We also write $\#A:=|A|$ for the cardinality of a
finite set $A$.

\begin{lemma}
    \label{lemma:cov}
    For any $\beta\in[0,\infty)$, $n\in\Z_{\geq 0}$,
    and $a,b\in\Z^2$,
    we have
    \begin{equation}
        \Cov_{n,\beta}[a;b]
        =
        \E_{n,\beta/2}[
            \#\{
                \eta\in C(\Pi):\text{$\eta$ surrounds both $\F_a$ and $\F_b$}
            \}
        ] .
    \end{equation}
\end{lemma}

\begin{proof}
    By definition of the height function,
    \begin{multline}
        h(\F_a)
        =\#\{\eta\in C(\Pi):
        \text{$\eta$ is counterclockwise and surrounds $\F_a$}\}
        \\
        -\#\{\eta\in C(\Pi):
        \text{$\eta$ is clockwise and surrounds $\F_a$}\}.
    \end{multline}
    Proposition~\ref{propo:flips} says that, conditionally on $\bar C(\Pi)$,
    the orientations of the cycles of length at least four are independent fair
    signs.
    Therefore all cross terms vanish when the product
    $h(\F_a)h(\F_b)$ is conditionally averaged, while a cycle contributes one
    precisely when it surrounds both faces.  Taking expectations gives the
    claimed identity.
\end{proof}

\part{Derivation of the inequalities}
\label{part:inequalities}
\section{The Ginibre inequality}
\label{sec:Ginibre}

The purpose of this section is to cast the Ginibre inequality in a useful form.
For the sake of generality, let $G=(V,E)$ denote
an arbitrary finite simple graph, and let $\vec E$ denote the set of directed edges
corresponding to edges in $E$.
A \emph{family of coupling constants} is a function $J:E\to[0,\infty)$,
and the corresponding XY model is given by
\begin{equation}
    \langle f\rangle_{G,\{J\}}
    :=
    \frac{1}{Z_{G,\{J\}}}
    \int
    \diff\sigma
    f(\sigma)
    \exp
    \sum_{\{x,y\}\in E}
    \frac{J_{xy}}2(\sigma_x\bar\sigma_y+\bar\sigma_x\sigma_y).
\end{equation}
Here $\diff\sigma$ denotes the Haar measure on $(\S^1)^V$;
the notation with curly brackets is used to distinguish this parametrisation
from the one with radii.
Of importance to us is the following almost immediate corollary of the Ginibre
inequality.

\begin{lemma}
    \label{lemma:traditional_ginibre}
    For any $J,A,B,C:E\to[0,\infty)$, we have
    \begin{equation}
        Z_{G,\{J\}}\cdot Z_{G,\{J+A+B+C\}}
        \geq
        Z_{G,\{J+A\}}\cdot Z_{G,\{J+B\}}.
    \end{equation}
\end{lemma}

\begin{proof}
    The proof is standard, and we shall briefly describe it.
    For any $X:E\to[0,\infty)$, write
    $Z_{G,\{J+X\}}$ as
    \begin{equation}
        Z_{G,\{J+X\}}=Z_{G,\{J\}}\left\langle
        \textstyle
        \exp
        \sum_{\{x,y\}\in E}
        \frac{X_{xy}}2(\sigma_x\bar\sigma_y+\bar\sigma_x\sigma_y)
        \right\rangle_{G,\{J\}}.
    \end{equation}
    The desired inequality now follows by writing each partition function in
    this form, expanding the exponentials,
    and applying the Ginibre inequality term-by-term.
\end{proof}

\begin{corollary}
    \label{cor:gin}
    Consider a fixed radius function $r:V\to[0,\infty)$.
    Then we have
    \begin{equation}
        \label{Gin.a}
        \tag{Gini}
        z_{r^2}(\tau_1+\tau_2)
        \geq
        z_{r^2}(\tau_1)\cdot z_{r^2}(\tau_2)
    \end{equation}
    \stepcounter{equation}
    for any $\tau_1,\tau_2:V\to[0,\infty)$ with disjoint supports.
\end{corollary}

This is the analogue of~\cite[Equation~36]{Frohlich_1982_TrivialityLphd4Theories}.
The inequality may also be interpreted as a Fortuin--Kasteleyn--Ginibre lattice condition,
but we have no use for this interpretation in this article.

\begin{proof}[Proof of Corollary~\ref{cor:gin}]
    After we expand the definitions and write each partition function
    in terms of a traditional family of coupling constants $J:E\to[0,\infty)$,
    the inequality is a direct consequence of Lemma~\ref{lemma:traditional_ginibre}.
\end{proof}

We also state another corollary which may be viewed as a \emph{monotonicity in domains}
property that has been used extensively in the analysis of the random-cluster model
with parameters $p\in[0,1]$ and $q\in[1,\infty)$; see, for
example,~\cite[Section~4.2]{Grimmett_2006_RandomClusterModel}.

\begin{corollary}
    Let $\tilde G=(\tilde V,\tilde E)$ denote a finite simple graph
    endowed with a radius function $\tilde r:\tilde V\to[0,\infty)$.
    Let $V\subset \tilde V$, let $G$ denote the subgraph of $\tilde G$ induced
    by $V$, and let $r$ denote the restriction of $\tilde r$ to $V$.
    Fix $\tau:V\to[0,\infty)$, and extend this function to
    a function $\tilde\tau:\tilde V\to[0,\infty)$ by setting $\tilde\tau(x):=0$
    for $x\in\tilde V\setminus V$.
    Then we have
    \begin{equation}
        \label{Mon}
        \tag{Mono}
        z_{\tilde G,\tilde r^2}(\tilde\tau)
        \leq
        z_{G,r^2}(\tau).
    \end{equation}
    \stepcounter{equation}
\end{corollary}

\begin{proof}
    Setting $\tilde\tau_2:=(\tilde r^2/2)\cdot 1_{\tilde V\setminus V}$, we have
    \begin{equation}
        z_{\tilde G,\tilde r^2}(\tilde\tau)
        =
        \frac{
            Z_{\tilde G,\tilde r^2}(\tilde\tau)
        }{
            Z_{\tilde G,\tilde r^2}
        }
        \leq
        \frac{
            Z_{\tilde G,\tilde r^2}(\tilde\tau+\tilde\tau_2)
        }{
            Z_{\tilde G,\tilde r^2}(\tilde\tau_2)
        }
        =\frac{
            Z_{G,r^2}(\tau)
        }
        {Z_{G,r^2}}
        =z_{G,r^2}(\tau),
    \end{equation}
    where the inequality follows from Corollary~\ref{cor:gin}.
\end{proof}

\section{The strip estimate}

\label{sec:strip}

We introduce some new notations illustrated by Figure~\ref{fig:strip_lemma_3}:
for any $n\in\Z_{\geq 0}$,
let $\partial\Lambda_n\subset\Lambda_{n+1}\setminus\Lambda_n$
denote the set of vertices adjacent to $\Lambda_n$,
and write $e_1:=(1,0)\in\Z^2$ and $e_2:=(0,1)\in\Z^2$.
Consider some integer $n\in\Z_{\geq 1}$.
Define first the \emph{strip} $S_n\subset\Lambda_n$ by
\begin{equation}
    S_n:=([-n,n]\cap\Z)\times ([-n,-1]\cap\Z).
\end{equation}
Its \emph{top} $T_n\subset S_n$ and \emph{bottom} $B_n\subset\partial\Lambda_n$ are defined by
\begin{equation}
    T_n:=([-n,n]\cap\Z)\times\{-1\};
    \quad
    B_n:=([-n,n]\cap\Z)\times\{-n-1\}.
\end{equation}

\begin{lemma}[Strip estimate]
    \label{lemma:strip_new}
    Let $\beta\in[0,\infty)$.
    Then as $n\to\infty$, we have
    \begin{equation}
        \stripXY{n}
        :=
        \max_{(x,y)\in T_n\times B_n}
        \inf_{m>n}
        \sum_{\omega\in\Gamma_{xy}[S_n]}
        \int
        z_{m,\beta}(\tau)
        \diff\rho_{\omega^*}(\tau)
        \geq e^{-\massXY{\beta} n+o(n)}.
    \end{equation}
    Write $(x_n,y_n)\in T_n\times B_n$ for some pair maximising the infimum
    in the display.
\end{lemma}

Note that $z_{m,\beta}(\tau)$ is decreasing in $m\in\Z_{>n}$
due to~\eqref{Mon},
and therefore the infimum may also be interpreted as a limit.

\begin{proof}
    Fix $n\in\Z_{\geq 1}$ and $m\in\Z_{>n}$.
    Let $u:=(0,0)\in\Z^2$
    and write $r\equiv\sqrt\beta$.
    The proof has two ingredients:
    the BFS random walk expansion of the two-point function,
    in particular~\eqref{2p.integrated},
    and the \emph{trivial bound}
    \begin{equation}
        \tag{TB}
        \label{TB}
        \langle\sigma_x\bar\sigma_y\rangle\leq 1,
    \end{equation}
    \stepcounter{equation}
    which holds true on any finite graph and for any family of coupling constants.
    Let $F\subset \partial\Lambda_n$ denote the set containing the four points
    of the form $(\pm (n+1),0)$ and $(0,\pm (n+1))$.
    We start the proof by an obvious equality and an expansion of each two-point function:
    \begin{equation}
        4\langle\sigma_{(0,0)}\bar\sigma_{(n+1,0)}\rangle_{m,\beta}
        =
        \sum_{v\in F}
        \langle\sigma_v\bar\sigma_u\rangle_{m,\beta}
        \stackrel{\eqref{2p.integrated}}=
        \frac2{r^2}
        \sum_{v\in F}
        \sum_{\omega\in\Gamma_{uv}}
        \int z_{m,\beta}(\tau)\diff\rho_\omega(\tau).
    \end{equation}

    The remainder of the proof consists of two steps.
    \begin{enumerate}
        \item We start by splitting the random walk $\omega$ into three parts;
            see Figure~\ref{fig:strip_lemma_3}.
            The vertex $y\in\Z^2$ is the first vertex in $\partial\Lambda_n$ that $\omega$
            hits. Since the system is invariant under $\pi/2$ rotation,
            we may assume without loss of generality that $y\in B_n$.
            We then reverse the segment of $\omega$ from $u$ to $y$,
            and let $x\in\Z^2$ denote the last vertex in $T_n$ that this reversed segment
            visits before exiting $S_n$.
            Thus, the walk $\omega$ visits the vertices $u,x,y,v$ in this order.
        \item
            We are interested in the middle line segment as it is the line
            segment appearing in the statement of the lemma.
            We must somehow erase the other two line segments in order to obtain the
            desired inequality.
            Each line segment is erased by applying the two-point identity and
            the trivial bound.
    \end{enumerate}

    \begin{figure}
        \includegraphics{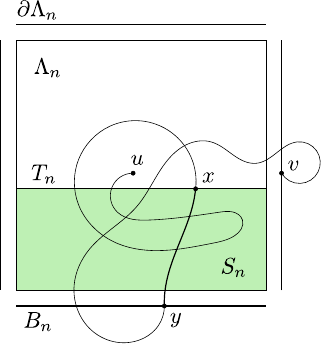}
        \caption{The decomposition of the BFS random walk into three segments}
        \label{fig:strip_lemma_3}
    \end{figure}

    To simplify the analysis, we first split $\omega$ into two segments,
    erase one,
    split the remainder, then again erase one of the two segments.
    Splitting on the vertex $y$ and applying the symmetry yields
    \begin{multline}
        4\langle\sigma_{(0,0)}\bar\sigma_{(n+1,0)}\rangle_{m,\beta}
        =
        \frac8{r^2}
        \sum_{v\in F}
        \sum_{y\in B_n}
        \sum_{\omega\in\Gamma_{uy}[\Lambda_n]}
        \sum_{\omega'\in\Gamma_{yv}}
        \int z_{m,\beta}(\tau)\diff\rho_{\omega^*\omega'}(\tau)
        \\
        \stackrel{\eqref{simplex}}=
        \frac8{r^2}
        \sum_{v\in F}
        \sum_{y\in B_n}
        \sum_{\omega\in\Gamma_{uy}[\Lambda_n]}
        \sum_{\omega'\in\Gamma_{yv}}
        \iint z_{m,\beta}(\tau_1+\tau_2)\diff\rho_{\omega^*}(\tau_1)\diff\rho_{\omega'}(\tau_2).
    \end{multline}
    Rearranging and applying the two-point identity yields equality with
    \begin{multline}
        \frac8{r^2}
        \sum_{v\in F}
        \sum_{y\in B_n}
        \sum_{\omega\in\Gamma_{uy}[\Lambda_n]}
        \int
        \left[
            \sum_{\omega'\in\Gamma_{yv}}
            \int
            \frac{z_{m,\beta}(\tau_1+\tau_2)}
            {z_{m,\beta}(\tau_1)}
            \diff\rho_{\omega'}(\tau_2)
        \right]
        z_{m,\beta}(\tau_1)
        \diff\rho_{\omega^*}(\tau_1)
        \\
        \stackrel{\eqref{2p.integrated}}=
        \frac8{r^2}
        \sum_{v\in F}
        \sum_{y\in B_n}
        \sum_{\omega\in\Gamma_{uy}[\Lambda_n]}
        \int
        \left[
            \frac{r^2}2
            \langle\sigma_v\bar\sigma_{y}\rangle_{m,r^2-2\tau_1}
        \right]
        z_{m,\beta}(\tau_1)
        \diff\rho_{\omega^*}(\tau_1)
        \\
        \stackrel{\eqref{TB}}\leq
        16
        \sum_{y\in B_n}
        \sum_{\omega\in\Gamma_{uy}[\Lambda_n]}
        \int
        z_{m,\beta}(\tau)
        \diff\rho_{\omega^*}(\tau).
    \end{multline}
    For the equality, it is important to observe that $\tau_1$ is almost
    everywhere supported on $\Lambda_n$, so that the radii in the denominator
    in~\eqref{2p.integrated} are really equal to $r$ and not to some reduced radius.
    After applying the trivial bound, the terms no longer depend on $v$
    so that the sum over $v$ may be replaced by a simple factor $|F|=4$.
    Thus, we have now established our intermediate inequality
    \begin{equation}
        \frac{1}{4}\langle\sigma_{(0,0)}\bar\sigma_{(n+1,0)}\rangle_{m,\beta}
        \leq
        \sum_{y\in B_n}
        \sum_{\omega\in\Gamma_{uy}[\Lambda_n]}
        \int
        z_{m,\beta}(\tau)
        \diff\rho_{\omega^*}(\tau).
    \end{equation}
    Summing over the vertex $x\in T_n$ described before,
    we obtain equality of the right-hand side with
    \begin{align}
        &
        \sum_{x\in T_n}
        \sum_{y\in B_n}
        \sum_{\omega'\in\Gamma_{u(x+e_2)}[\Lambda_n]}
        \sum_{\omega\in\Gamma_{xy}[S_n]}
        \int
        z_{m,\beta}(\tau)
        \diff\rho_{\omega'\omega^*}(\tau)
        \\
        &
        \stackrel{\eqref{simplex}}=
        \sum_{x\in T_n}
        \sum_{y\in B_n}
        \sum_{\omega\in\Gamma_{xy}[S_n]}
        \int
        \left[
            \int
            \sum_{\omega'\in\Gamma_{u(x+e_2)}[\Lambda_n]}
            \frac{z_{m,\beta}(\tau_1+\tau_2)}{z_{m,\beta}(\tau_2)}
            \diff\rho_{\omega'}(\tau_1)
        \right]
        z_{m,\beta}(\tau_2)
        \diff\rho_{\omega^*}(\tau_2)
        \\
        &\leq
        \sum_{x\in T_n}
        \sum_{y\in B_n}
        \sum_{\omega\in\Gamma_{xy}[S_n]}
        \int
        \left[
            \int
            \sum_{\omega'\in\Gamma_{u(x+e_2)}}
            \frac{z_{m,\beta}(\tau_1+\tau_2)}{z_{m,\beta}(\tau_2)}
            \diff\rho_{\omega'}(\tau_1)
        \right]
        z_{m,\beta}(\tau_2)
        \diff\rho_{\omega^*}(\tau_2)
        \\
        &
        \stackrel{\eqref{2p.integrated}}=
        \sum_{x\in T_n}
        \sum_{y\in B_n}
        \sum_{\omega\in\Gamma_{xy}[S_n]}
        \int
        \left[
            \frac{r^2}2
            \langle
            \sigma_u\bar\sigma_{x+e_2}
            \rangle_{m,r^2-2\tau_2}
        \right]
        z_{m,\beta}(\tau_2)
        \diff\rho_{\omega^*}(\tau_2)
        \\
        &
        \stackrel{\eqref{TB}}\leq
        \frac{r^2}2
        \sum_{x\in T_n}
        \sum_{y\in B_n}
        \sum_{\omega\in\Gamma_{xy}[S_n]}
        \int
        z_{m,\beta}(\tau)
        \diff\rho_{\omega^*}(\tau).
    \end{align}
    The first equality is a simple rearrangement.
    The first inequality is a direct consequence of the fact that
    $\Gamma_{u(x+e_2)}[\Lambda_n]\subset\Gamma_{u(x+e_2)}$.
    When applying~\eqref{2p.integrated}, it is again important to observe that $\tau_2$
    is supported on $S_n$ so that the radii dropping out are indeed $r$.
    Combining with the intermediate inequality and $r^2=\beta$ yields
    \begin{equation}
        \label{eq:final_ineq}
        \frac1{2\beta}\langle\sigma_{(0,0)}\bar\sigma_{(n+1,0)}\rangle_{m,\beta}
        \leq
        \sum_{x\in T_n}
        \sum_{y\in B_n}
        \sum_{\omega\in\Gamma_{xy}[S_n]}
        \int
        z_{m,\beta}(\tau)
        \diff\rho_{\omega^*}(\tau).
    \end{equation}
    The left-hand side is increasing in $m$ by the Ginibre inequality,
    and the right-hand side is decreasing in $m$ due to~\eqref{Mon}.
    Therefore the inequality remains true if we replace the graph $\Lambda_m$ by
    $\Z^2$ on the left, and if we take an infimum over $m>n$
    inside the integral
    on the right.
    The lemma follows by taking $m\to\infty$,
    maximising over $x$ and $y$,
    and observing that $|T_n|=|B_n|=(2n+1)$.
\end{proof}


\section{The loop estimate}
\label{sec:loop}

The purpose of this section is to cast the strip estimate into a useful form.
For $n\in\Z_{\geq 1}$, let $\GammaLoop{n}\subset\Gamma_{(0,0)(0,0)}^{(2)}$
denote the set of walks $\omega\in\Gamma$
such that:
\begin{itemize}
    \item The walk starts at $\omega_0=(0,0)$ and $\omega_1=e_1$,
    \item The walk ends at $\omega_{\ell(\omega)-1}=-e_1$ and $\omega_{\ell(\omega)}=(0,0)$,
    \item The walk does not hit the vertical axis $\{0\}\times\Z$
        other than at its starting and ending vertices, except that it hits
        the vertex $(0,n)$ precisely once.
\end{itemize}
See Figure~\ref{fig:new_combined}, \textsc{Right}
for an illustration of a loop in $\GammaLoop{2kn+1}$.

\begin{lemma}[Loop estimate]
    \label{lemma:BFS_loop_estimate}
    Let $\beta\in[0,\infty)$.
    Then
    for each $n,k\in\Z_{\geq 1}$,
    we have
    \begin{equation}
        \liminf_{m\to\infty}
        \sum_{\omega\in\GammaLoop{2kn+1}}
        \int
        z_{m,\beta}(\tau)
        \diff\rho_{\omega^*}(\tau)
        \geq
        c_\beta^{8n+6}\times \stripXY{n}^{4k};
        \qquad
        c_\beta:=\frac\beta2e^{-4\beta}
        .
    \end{equation}
    In particular, by setting $k=n$, sending $n$ to $\infty$,
    and applying the strip estimate, we get
    \begin{equation}
        \liminf_{m\to\infty}
        \sum_{\omega\in\GammaLoop{k_n+1}}
        \int
        z_{m,\beta}(\tau)
        \diff\rho_{\omega^*}(\tau)
        \geq
        e^{-2\massXY{\beta}k_n+o(k_n)}
    \end{equation}
    along the subsequence $(k_n)_{n\in\Z_{\geq 1}}$
    defined by $k_n:=2n^2$.
\end{lemma}

\begin{proof}
    The proof relies on the following inputs.
    \begin{itemize}
        \item If two walks $\omega$ and $\omega'$ are disjoint,
            then $\tau_1$ and $\tau_2$ have disjoint support almost everywhere in the product
            measure $\rho_{\omega}(\tau_1)\times\rho_{\omega'}(\tau_2)$.
            Therefore we may apply the Ginibre inequality~\eqref{Gin.a} for comparing
            normalised partition functions.
        \item The lower bound in the strip estimate (Lemma~\ref{lemma:strip_new})
            is the driving force behind the inequality we aim to prove.
            The paths appearing in the strip estimate are geometrically constrained
            to the strip. This is convenient: it will allow us to combine different
            paths and be certain that they are disjoint (except at the overlapping
            starting points and endpoints) so that the Ginibre inequality applies.
        \item Finally, we shall use the symmetries of the square lattice and the monotonicity
            in the graph~\eqref{Mon} to transport the strip estimate around the graph.
    \end{itemize}

    Fix $\beta\in[0,\infty)$
    and let $n\in\Z_{\geq 1}$.
    Define the \emph{alternative strip} $S':=([-2n,2n]\cap\Z)\times([-2n,-1]\cap\Z)$.
    Let $x':=-e_2$ and $y':=-(2n+1)e_2$;
    see Figure~\ref{fig:new_combined}, \textsc{Left}.
    We first prove the following claim.

    \begin{claim*}
        We have
        \begin{equation}
            \label{eq:left}
            \inf_{m>2n}
            \sum_{\omega\in\Gamma_{x'y'}[S']}
            \int
            z_{m,\beta}(\tau)
            \diff\rho_{\omega^*}(\tau)
            \geq\stripXY{n}^2.
        \end{equation}
    \end{claim*}

    \begin{proof}[Proof of the claim]
        \begin{figure}

            \hfill
            \begin{subfigure}{0.46\textwidth}
                \includegraphics{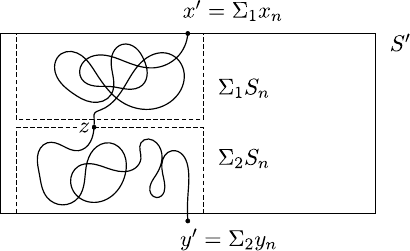}
            \end{subfigure}
            \hfill
            \begin{subfigure}{0.31\textwidth}
                \includegraphics{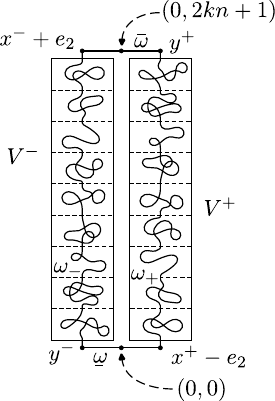}
            \end{subfigure}
            \hfill

            \caption{}
            \label{fig:new_combined}
        \end{figure}
        
        Let $S_n$, $x_n$, and $y_n$ be defined as in Lemma~\ref{lemma:strip_new}.
        We now define two symmetries of the square lattice which are illustrated by
        Figure~\ref{fig:new_combined},
        \textsc{Left}.
        Let $\Sigma_1$ denote the unique horizontal translation
        such that $\Sigma_1 x_n=x'$.
        Let $\Sigma_2$ denote the symmetry of the square lattice which:
        \begin{itemize}
            \item Is the composition of a translation with a reflection through the
                $y$-axis $\{0\}\times\R$,
            \item Maps $x_n$ to $\Sigma_2x_n=\Sigma_1y_n=:z$,
            \item Maps $y_n$ to $\Sigma_2y_n=y'$.
        \end{itemize}
        The definitions imply that
        $\Sigma_1 S_n$ and $\Sigma_2 S_n$ are disjoint
        and contained in $S'$.

        Lemma~\ref{lemma:strip_new}, monotonicity in domains~\eqref{Mon},
        and invariance of the XY model under the symmetries of the square lattice
        jointly imply that for $i\in\{1,2\}$ we have
        \begin{equation}
            \label{eq:applied_strip}
            \sum_{\omega\in\Gamma_{\Sigma_ix_n\Sigma_iy_n}[\Sigma_iS_n]}
            \int
            z_{m,\beta}(\tau)
            \diff\rho_{\omega^*}(\tau)
            \geq\stripXY{n}
        \end{equation}
        as soon as $m\in\Z_{\geq 0}$ is so large that $\Sigma_iS_n\subset\Lambda_m$.
        Now focus on the expression on the left in~\eqref{eq:left}.
        By splitting on the first vertex $\tilde z$ outside $\Sigma_1S_n$
        that the path hits,
        we get
        \begin{align}
            &\sum_{\omega\in\Gamma_{x'y'}[S']}
            \int
            z_{m,\beta}(\tau)
            \diff\rho_{\omega^*}(\tau)
            \\
            &\qquad
            =
            \sum_{\tilde z\in S'\setminus \Sigma_1S_n}
            \sum\nolimits_{\omega\in\Gamma_{x'\tilde z}[\Sigma_1S_n]}
            \sum\nolimits_{\tilde\omega\in\Gamma_{\tilde zy'}[S']}
            \int
            z_{m,\beta}(\tau)
            \diff\rho_{\omega^*\tilde\omega^*}(\tau)
            \\
            &\qquad
            \geq
            \sum\nolimits_{\omega\in\Gamma_{x'z}[\Sigma_1S_n]}
            \sum\nolimits_{\tilde\omega\in\Gamma_{zy'}[\Sigma_2S_n]}
            \int
            z_{m,\beta}(\tau)
            \diff\rho_{\omega^*\tilde\omega^*}(\tau)
            \\
            &\qquad
            \geq
            \left[
                \sum_{\omega\in\Gamma_{x'z}[\Sigma_1S_n]}
                \int
                z_{m,\beta}(\tau_1)
                \diff\rho_{\omega^*}(\tau_1)
            \right]\left[
                \sum_{\tilde\omega\in\Gamma_{zy'}[\Sigma_2S_n]}
                \int
                z_{m,\beta}(\tau_2)
                \diff\rho_{\tilde\omega^*}(\tau_2)
            \right]
            \\
            &\qquad
            \geq
            \stripXY{n}^2.
        \end{align}
        This first inequality holds true because we sum over a
        smaller set of pairs $(\omega,\tilde\omega)$ of paths,
        the second inequality follows from~\eqref{simplex} and~\eqref{Gin.a},
        and the final
        inequality is just~\eqref{eq:applied_strip} applied to both factors.
        This proves the claim.
        \renewcommand\qedsymbol{}
    \end{proof}

    Fix $k\in\Z_{\geq 1}$ and define
    \begin{multline}
        V^-:=([-4n-1,-1]\cap\Z)\times([1,2kn]\cap\Z);
        \\
        x^-:=(-2n-1,2kn);
        \quad
        y^-:=(-2n-1,0);
    \end{multline}
    see Figure~\ref{fig:new_combined}, \textsc{Right}.
    By writing $V^-$ as $k$ disjoint translates of $S'$
    and reasoning as in the proof of the claim, we observe that
    \begin{equation}
        \label{eq:longestimate}
        \inf_{m>4kn}
        \sum_{\omega_-\in\Gamma_{x^-y^-}[V^-]}
        \int
        z_{m,\beta}(\tau)
        \diff\rho_{\omega_-^*}(\tau)
        \geq\stripXY{n}^{2k};
    \end{equation}
    see Figure~\ref{fig:new_combined}, \textsc{Right}.
    The same holds true if we replace $-$ by $+$
    where
    \begin{multline}
        V^+:=([1,4n+1]\cap\Z)\times([1,2kn]\cap\Z);
        \\
        x^+:=(2n+1,1);
        \quad
        y^+:=(2n+1,2kn+1).
    \end{multline}
    In order to turn the two line segments $\omega_+$
    and $\omega_-$ into a loop in $\GammaLoop{2kn+1}$,
    we must connect the two line segments at the top and bottom.
    We make those connections in the easiest possible way:
    by adding straight line segments of minimal length,
    see Figure~\ref{fig:new_combined}, \textsc{Right}.
    Let $\bar\omega$ denote the straight line segment
    from $y^+$ to $x^-+e_2$,
    and let $\ubar\omega$ denote the straight line segment from
    $y^-$ to $x^+-e_2$.
    This choice of line segments is such that for any
    $\omega_+$ and $\omega_-$ in their relevant sets of paths:
    \begin{itemize}
        \item The loop $\omega_+^* \bar\omega \omega_-^* \ubar\omega$
            lies in $\GammaLoop{2kn+1}$ (up to indexation),
        \item The four line segments $\omega_+^*$,
            $\bar\omega$, $\omega_-^*$, and $\ubar\omega$ are disjoint.
    \end{itemize}
    Note that for any vertex $x\in\Lambda_m$, we have
    \begin{equation}
        \label{eq:siteestimate}
        \int_0^\infty  z_{m,\beta}(\tau\cdot 1_x) \diff\tau\geq
        \frac\beta2e^{-4\beta}=:c_\beta.
    \end{equation}
    Conclude that
    \begin{multline}
        \inf_{m>4kn}
        \sum_{\omega\in\GammaLoop{2kn+1}}
        \int
        z_{m,\beta}(\tau)
        \diff\rho_{\omega^*}(\tau)
        \\
        \geq
        \inf_{m>4kn}
        \sum_{\omega_+\in\Gamma_{x^+y^+[V^+]}}
        \sum_{\omega_-\in\Gamma_{x^-y^-}[V^-]}
        \int
        z_{m,\beta}(\tau)
        \diff\rho_{\omega_+^* \bar\omega \omega_-^* \ubar\omega}(\tau)
        \\
        \stackrel{\eqref{Gin.a}}\geq
        c_\beta^{8n+6}\times\stripXY{n}^{4k}.
    \end{multline}
    The first inequality holds true because we effectively
    sum over a smaller set of paths.
    The second inequality holds true because we
    may decompose the loop into its four disjoint segments with the Ginibre
    inequality,
    using~\eqref{eq:longestimate} for the vertical segments
    and~\eqref{eq:siteestimate} for each vertex visited
    by the horizontal segments.
    This completes the proof.
\end{proof}

\section{Lower bound on the height function correlation length}
\label{sec:heightcorrlength}

\begin{lemma}
    \label{lemma:hierlemma}
    For any $\beta\in[0,\infty)$ and $n,n',m\in\Z_{\geq 0}$
    with $n\leq n'<m$, we have
    \begin{equation}
        \label{eq:finalbound}
        \Cov_{m,\beta}[(0,0);(0,n)]
        \geq
        \sum_{\omega\in\GammaLoop{n'+1}}
        \int
        z_{m,\beta}(\tau)\diff\rho_{\omega^*}(\tau).
    \end{equation}
\end{lemma}

\begin{proof}
    Let $x= (0,0)$.
    Consider a sample from $\P_{m,\beta/2}$ and explore the outgoing Poisson edge
    from $x$ with the highest time, if such an edge exists.
    If this edge exists, then run Strategy~\ref{strategy:naive} until returning back to $x$.
    This yields a slight modification of~\eqref{2p.naive}: we get
    \begin{equation}
        1
        =
        \langle\sigma_x\bar\sigma_x\rangle_{m,\beta}
        =
        z_{m,\beta}((\beta/2)\cdot 1_x)
        +
        \sum_{\omega\in\Gamma_{xx}^{(2)}}
        \int
        z_{m,\beta}(\tau)\diff\rho_{\omega^*}(\tau).
    \end{equation}
    The first term on the right corresponds to the event
    that no Poisson edge is incident to $x$.
    The probability that the path $\omega$ so explored lies in
    $\GammaLoop{n'+1}$ is precisely equal to
    \begin{equation}        \sum_{\omega\in\GammaLoop{n'+1}}
        \int
        z_{m,\beta}(\tau)\diff\rho_{\omega^*}(\tau).
    \end{equation}
    Let $\Pi'\subset\Pi$ denote the corresponding set of Poisson edges traversed.
    Then $\Pi'$ is the disjoint union of cycles in $C(\Pi)$
    because of Proposition~\ref{pro:loop_decomp_lemma}.
    Since $\omega$ hits precisely two vertices of the vertical axis $\{0\}\times\Z$
    in order to pass from one side of this axis to the other,
    namely at $(0,0)$ and $(0,n'+1)$,
    there is some cycle $\eta\in C(\Pi)$ with the same properties.
    In particular, this cycle surrounds both $\F_{(0,0)}$ and $\F_{(0,n)}$.
    This finishes the proof because we have now established that
    \begin{multline}
        \Cov_{m,\beta}[(0,0);(0,n)]
        =
        \E_{m,\beta/2}[
            \#\{\eta\in C(\Pi):
            \text{$\eta$ surrounds both $\F_{(0,0)}$ and $\F_{(0,n)}$}\}
        ]
        \\
        \geq
        \P_{m,\beta/2}[\{
            \exists\eta\in C(\Pi):
            \text{$\eta$ surrounds both $\F_{(0,0)}$ and $\F_{(0,n)}$}
        \}]
        \\
        \geq
        \sum_{\omega\in\GammaLoop{n'+1}}
        \int
        z_{m,\beta}(\tau)\diff\rho_{\omega^*}(\tau).
    \end{multline}
    The first equality is due to Lemma~\ref{lemma:cov}.
\end{proof}

\begin{corollary}
    \label{cor:lower_bound_height_corr}
    For any $\beta\in[0,\infty)$, we have, as $n\to\infty$,
    \begin{equation}
        \Cov_{\infty,\beta}[(0,0);(0,n)]
        \geq
        e^{-2\massXY{\beta}n+o(n)}.
    \end{equation}
    In particular, $\massHeight{\beta}\leq 2\massXY{\beta}$.
\end{corollary}

\begin{proof}
    For fixed $n\in\mathbb Z_{\geq 1}$, we let $n'$ denote the smallest
    integer larger than or equal to $n$ that can be written in the form $2k^2$ for some integer $k$.
    By Lemmas~\ref{lemma:BFS_loop_estimate} and~\ref{lemma:hierlemma}
    and the fact that the covariance matrix is non-decreasing in $m$,
    we have
    \begin{multline}
        \Cov_{\infty,\beta}[(0,0);(0,n)]
        \geq
        \liminf_{m\to\infty}
        \sum_{\omega\in\GammaLoop{n'+1}}
        \int
        z_{m,\beta}(\tau)\diff\rho_{\omega^*}(\tau)
        \\
        \geq
        e^{-2\massXY{\beta}n'+o(n')}
        =
        e^{-2\massXY{\beta}n+o(n)}.
    \end{multline}
    For the last step one simply observes that $n'/n\to 1$
    as $n\to\infty$.
\end{proof}

\section{Lower bound on the XY model correlation length}
\label{sec:lowerboundXYcorr}

Recall that $h$ denotes the height function,
which maps faces to integers.
For any integer $z\in\Z$, let $\Sign{z}\in\{-1,0,+1\}$ denote the
\emph{sign} of $z$.
Similarly, we write $\Sign{h}$ for the random function which assigns
signs to the faces of the square lattice.
We write $\SigCov$ for $\Cov$ except that the height function
$h$ is replaced by $\Sign{h}$ in the definition,
that is,
\begin{equation}
    \SigCov_{n,\beta}:\Z^2\times\Z^2\to[0,1],\,
    (a,b)
    \mapsto
    \E_{n,\beta/2}[\Sign{h(\F_a)}\Sign{h(\F_b)}].
\end{equation}
This function is non-decreasing in $n$~\cite[Lemma~4.1]{Lammers_2026_DichotomyTheoryHeight},
and we let \(\SigCov_{\infty,\beta}:=
    \lim_{n\to\infty}
    \SigCov_{n,\beta}
\).
The following lemma asserts that the covariance matrix for the sign of $h$
yields an alternative definition for the mass.

\begin{lemma}[{\cite[Theorem~2]{Lammers_2026_DichotomyTheoryHeight}}]
    \label{lemma:altmassdef}
    For any $\beta\in [0,\infty)$ we have
    \begin{equation}
        \massHeight{\beta}=
        \liminf_{k\to\infty}-\frac1k\log\SigCov_{\infty,\beta}[(0,0);(0,k)].
    \end{equation}
\end{lemma}

This alternative definition of the mass is more useful
to work with when lower bounding the two-point function of the XY model.

\begin{lemma}
    For any $\beta\in [0,\infty)$, $n\in\Z_{\geq 0}$, and $a,b\in\Z^2$, we have
    \begin{equation}
        \P_{n,\beta/2}[\{\exists\eta\in C(\Pi):\text{$\eta$ surrounds $\F_a$ and $\F_b$}\}]
        \geq
        \SigCov_{n,\beta}[a;b].
    \end{equation}
\end{lemma}

\begin{proof}
    Let $N$ denote the number of cycles in $C(\Pi)$ surrounding both $\F_a$ and $\F_b$.
    Note that Proposition~\ref{propo:flips} implies that
    conditional on $\{N=0\}$,
    the law of the signs of $h$ at $\F_a$ and $\F_b$ is invariant under
    tossing two independent fair coins to decide if the signs at the two
    faces should be inverted or not.
    In particular,
    \begin{equation}
        \E_{n,\beta/2}[1_{\{N=0\}}\Sign{h(\F_a)}\Sign{h(\F_b)}]=0.
    \end{equation}
    This immediately implies the lemma.
\end{proof}

\begin{lemma}
    For any $\beta\in [0,\infty)$, $n\in\Z_{\geq 0}$, and any distinct $x,y\in\Lambda_n$, we have
    \begin{equation}
        2\langle\sigma_x\bar\sigma_y\rangle^2_{n,\beta}
        \geq
        \P_{n,\beta/2}[\{\exists\eta\in C(\Pi):x,y\in V(\eta)\}].
    \end{equation}
\end{lemma}

\begin{proof}
    Recall that $a_x(\eta)$ denotes the activation time of a cycle
    $\eta$.
    Suppose that we start an exploration at $x$ of type Strategy~\ref{explo:II} until
    the total local time removed at $x$
    reaches some constant $s\in [0,\beta/2]$.
    Write $\omega$ for the exploration path so obtained.
    Recall from~\eqref{2p.anytime} that the corresponding expansion is
    \begin{equation}
        1
        =
        \langle\sigma_x\bar\sigma_x\rangle_{\beta}
        =
        \sum_{\omega\in\Gamma_{xx}}\int_{\{\tau_x\leq s\}} z_{\beta}(\tau\vee 1_x\cdot
        s)\diff\rho_{\omega^*}(\tau).
    \end{equation}
    Recall from Proposition~\ref{pro:loop_decomp_lemma} that if $C(\Pi)$ contains a cycle which
    hits both $x$ and $y$ and whose activation time at $x$
    is at least $\beta/2-s$, then $\omega$ will certainly hit $y$.
    Thus, we get
    \begin{multline}
        \label{eq:LHSSS}
        \sum_{\omega\in\Gamma_{xx}:\: y\in\omega}\int_{\{\tau_x\leq s\}}
        z_{\beta}(\tau\vee 1_x\cdot s)\diff\rho_{\omega^*}(\tau)
        =
        \P_{n,\beta/2}[\{y\in\omega\}]
        \\
        \geq
        \P_{n,\beta/2}[\{\exists\eta\in C(\Pi):x,y\in V(\eta),\, a_x(\eta)\geq \beta/2-s\}]=:B(s).
    \end{multline}
    We would like to lose the dependence on the activation time.
    By Proposition~\ref{pro:time_inversion}, we have
    \begin{equation}
        B(s)+B(\beta/2-s)\geq \P_{n,\beta/2}[\{\exists\eta\in C(\Pi):x,y\in V(\eta)\}].
    \end{equation}
    By integrating $s$ from $0$ to $\beta/2$ on the left-hand side of~\eqref{eq:LHSSS},
    we get
    \begin{equation}
        \frac2{\beta}
        \sum_{\omega\in\Gamma_{xx}:\: y\in\omega}\int z_{\beta}(\tau)\diff\rho_{\omega}(\tau)
        \geq \frac12  \P_{n,\beta/2}[\{\exists\eta\in C(\Pi):x,y\in V(\eta)\}].
    \end{equation}
    The left-hand side may be rewritten as
    \begin{multline}
        \frac2{\beta}
        \sum_{\omega\in\Gamma_{xy}^{(1)}}
        \int
        \left[
            \sum_{\tilde\omega\in\Gamma_{yx}}
            \int z_{\beta-2\tau}(\tilde\tau)\diff\rho_{\tilde\omega}(\tilde\tau)
        \right]
        z_{\beta}(\tau)\diff\rho_{\omega^*}(\tau)
        \\
        \stackrel{\eqref{2p.integrated}}=
        \frac2{\beta}
        \sum_{\omega\in\Gamma_{xy}^{(1)}}
        \int
        \left[
            \frac{\sqrt{\beta (\beta-2\tau_x)}}{2}
            \langle\sigma_x\bar\sigma_y\rangle_{n,\beta-2\tau}
        \right]
        z_{\beta}(\tau)\diff\rho_{\omega^*}(\tau)
        \\
        \stackrel{\eqref{Gin.a}}\leq
        \langle\sigma_x\bar\sigma_y\rangle_{n,\beta}
        \sum_{\omega\in\Gamma_{xy}^{(1)}}
        \int
        z_{\beta}(\tau)\diff\rho_{\omega^*}(\tau)
        \stackrel{\eqref{2p.naive}}=
        \langle\sigma_x\bar\sigma_y\rangle_{n,\beta}^2
    \end{multline}
    The lemma follows by combining this inequality with the previous display.
\end{proof}

\begin{lemma}
    \label{lemma:lower_bound_xy_corr}
    For any $\beta\in [0,\infty)$, we have $\massXY{\beta}\leq\massHeight{\beta}/2$.
\end{lemma}

\begin{proof}
    Suppose that $\massXY{\beta}>\massHeight{\beta}/2+\epsilon$
    for some $\epsilon>0$ in order to derive a contradiction.
    For $m\in\Z_{\geq 0}$, let
    \begin{equation}
        Q_-^m:=\{(0,k):k\in\Z_{\leq 0}\};
        \qquad
        Q_+^m:=\{(0,k):k\in\Z_{\geq m}\}.
    \end{equation}
    If a cycle $\eta$ surrounds both $\F_{(0,0)}$ and $\F_{(0,m)}$,
    then it must intersect both $Q_-^m$ and $Q_+^m$.
    By the previous two lemmas, we get
    \begin{equation}
        2\sum_{x\in Q_-^m,\, y\in Q_+^m}\langle
        \sigma_x\bar\sigma_y
        \rangle_{\Z^2,\beta}^2
        \geq
        \SigCov_{\infty,\beta}[(0,0);(0,m)].
    \end{equation}
    By our assumption we have, for $m$ large enough,
    \begin{equation}
        2\sum_{x\in Q_-^m,\, y\in Q_+^m}\langle
        \sigma_x\bar\sigma_y
        \rangle_{\Z^2,\beta}^2
        \leq
        2\sum_{i,j\in\Z_{\geq 0}}
        e^{-(m+i+j)(\massHeight{\beta}+2\epsilon)}
        =
        e^{-(\massHeight{\beta}+2\epsilon)m}\cdot C_{\beta,\epsilon}
    \end{equation}
    where the constant is finite and depends only on $\beta$ and $\epsilon$.
    At the same time Lemma~\ref{lemma:altmassdef} implies that
    for some subsequence $(k_m)_m\subset\Z_{\geq 0}$,
    and for $m$ sufficiently large, we have
    \begin{equation}
        \SigCov_{\infty,\beta}[(0,0);(0,k_m)]
        \geq
        e^{-(\massHeight{\beta}+\epsilon)k_m}.
    \end{equation}
    The previous three displays imply that for all $m$ sufficiently large, we have
    \begin{equation}e^{-(\massHeight{\beta}+\epsilon)k_m}
        \leq
        e^{-(\massHeight{\beta}+2\epsilon)k_m}\cdot C_{\beta,\epsilon}
    \end{equation}
    which is clearly a contradiction
    since the constant does not depend on $m$.
\end{proof}

\section{Proof of Theorem~\ref{thm2}}
\label{sec:corproof}

\begin{figure}
    \includegraphics[scale=0.5]{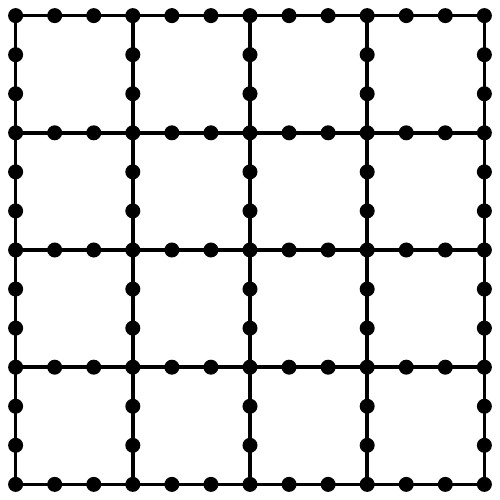}
    \caption{The graph $A_{2,3}$}
    \label{fig:a_2_3}
\end{figure}

\begin{proof}[Proof of Theorem~\ref{thm2}]
    The duality between the Villain model and the discrete Gaussian model is
    understood throughout this proof in finite volume.  Its first major use was in
    the classical work of Fr\"ohlich and Spencer
    \cite{FrohlichSpencer_1981_KosterlitzThoulessTransitionTwodimensional},
    although we do not claim that this is the first article in which the
    duality itself appears.  For a modern treatment, see Garban and
    Sep\'ulveda~\cite{GarbanSepulveda_2023_QuantitativeBoundsVortex}.

    Fix $\beta\in[0,\infty)$ and $n\in\Z_{\geq 1}$.  For each
    $m\in\Z_{\geq 1}$, let $A_{n,m}$ be the finite graph obtained from
    $\Lambda_n$ by replacing every edge by $m$ edges linked in series (cf. Figure~\ref{fig:a_2_3}).  One may
    choose an inverse temperature $\beta_m$ such that, as $m\to\infty$, the XY
    model on $A_{n,m}$ at inverse temperature $\beta_m$, restricted to the
    vertices of $\Lambda_n$, converges to the Villain model in $\Lambda_n$ at
    inverse temperature $\beta$.  Under the finite-volume duality, the
    corresponding height function converges to the discrete Gaussian model in
    the same finite volume.

    For each fixed pair $(n,m)$, the Poissonian exploration and the loop erasure
    decomposition developed above may be performed on $A_{n,m}$.  The cycles in
    the resulting decomposition are of two types.  There are microscopic
    two-edge cycles contained in the chains which replace the original edges,
    and there are macroscopic cycles which, after every such chain is contracted
    to a single edge, project to cycles of $\Lambda_n$.  The microscopic cycles
    surround no face, although they remain part of the finite-volume current and
    may therefore contribute to observables such as the two-point function.

    The key point in passing to the limit is that the projected macroscopic
    cycles form a tight family as $m\to\infty$, for each fixed $n$.
    Consequently, the $m\to\infty$ limit of these macroscopic cycles is well-defined,
    and all of the above arguments may be applied directly to this limiting family of
    macroscopic cycles. 
    This enables us to extend all results to the Villain model, which establishes
    Theorem~\ref{thm2}.
\end{proof}

\addtocontents{toc}{\protect\addvspace{\medskipamount}}
\section*{\phantomsection Acknowledgements}
\addcontentsline{toc}{section}{Acknowledgements}

The author is grateful to
Paul Dario, Hugo Duminil-Copin,
Jürg Fröhlich,
Trishen Gunaratnam, Romain Panis, Ron Peled, and Junchen Rong
for stimulating suggestions and discussions regarding the XY model in general
and this work in particular.
This project has received funding from the European Research Council
(ERC) under the European Union's Horizon 2020 research and innovation programme
(grant agreement No. 757296).
During the writing of this work, the author learned that Diederik van
Engelenburg and Marcin Lis are independently preparing a manuscript demonstrating
that polynomial decay of the two-point function in the two-dimensional XY model
implies delocalisation of the dual height function.

\defbibheading{bibintoc}[\refname]{%
    \section*{\phantomsection #1}%
    \addcontentsline{toc}{section}{#1}}
\printbibliography[heading=bibintoc]

\end{document}